\documentclass[reqno]{amsart}
\usepackage{amsmath, amssymb, latexsym, amsthm, amsfonts, mathrsfs,
  amscd,color, bbm, tikz-cd, hyperref, xfrac, faktor}
\usepackage{scalerel,stackengine}
\usepackage{tikz-cd}
\stackMath
\newcommand\reallywidehat[1]{%
\savestack{\tmpbox}{\stretchto{%
  \scaleto{%
    \scalerel*[\widthof{\ensuremath{#1}}]{\kern-.6pt\bigwedge\kern-.6pt}%
    {\rule[-\textheight/2]{1ex}{\textheight}}
  }{\textheight}%
}{0.5ex}}%
\stackon[1pt]{#1}{\tmpbox}%
}
\renewcommand{\eqref}[1]{(\ref{#1})}   

\numberwithin{equation}{section}

\theoremstyle{plain}
\newtheorem{theorem}{Theorem}[section]
\newtheorem{lemma}[theorem]{Lemma}
\newtheorem{corollary}[theorem]{Corollary}
\newtheorem{proposition}[theorem]{Proposition}
\newtheorem{conjecture}[theorem]{Conjecture}

\theoremstyle{definition}

\newtheorem{remark}[theorem]{Remark}

\theoremstyle{definition}

\newcommand{\Q}{{\mathbbm Q}}

\newcommand{\Z}{{\mathbbm Z}}
\newcommand{\C}{{\mathbbm C}}
\newcommand{\N}{{\mathbbm N}}
\newcommand{\F}{{\mathbbm F}}

\newcommand{\modp}{(\textnormal{mod }p)}

\title[On irreducibility of prefixed algebraic sets]{On irreducibility of prefixed algebraic sets in moduli spaces of prime degree polynomials}
\author[]{Niladri Patra\\ \tiny Research Associate - 1, Theoretical Statistics and Mathematics Department, \\ Indian Statistical Institute, Delhi Centre.} 
\address{ Office 122, Theoretical Statistics and Mathematics Department, Indian Statistical Institute, Delhi Centre, Katwaria Sarai, Hauz Khas, Delhi, India.}
\email{niladript@gmail.com}


\begin{document}

\maketitle

\begin{abstract}
    Consider the moduli space, $\mathcal{M}_{d}$, of degree $d > 2$ polynomials over $\C$, with a marked critical point. Given $k \geq 0,\; p$ an odd prime, we show that the set $\Sigma_{k,1,p}$ of conjugacy classes of degree $p$ polynomials, for which the marked critical point is strictly $(k,1)$-preperiodic, is an irreducible quasi-affine variety. Irreducibility of these sets was conjectured by Milnor, and has been proved for $p=3$ by Buff, Epstein and Koch.
     
    We prove that the subspaces of $\Sigma_{k,1,p}$, that arise by varying the ramification index of the marked critical point all the way up to the unicritical case, are all irreducible subvarieties. Finally, using the irreducibility of $\Sigma_{k,1,p}$ we give a new and short proof of the fact that the set of all unicritical points of $\Sigma_{k,1,p}$ form one Galois orbit under the action of absolute Galois group of $\Q$.
\end{abstract}

\section{Introduction}

Complex dynamics studies the behaviour of iterates of rational maps on the Riemann sphere. Since the time of Fatou and Julia, it has been known that the behaviour of critical points of a map, under iterations, dictates the behaviour of the dynamical system generated by that map. Thurston (\cite{10.1007/BF02392534}) showed that fixing the number of distinct images of critical points of a polynomial map, under the self-iterates of the map, determines the conjugacy class of the dynamical system, up to finitely many choices.

Let $\mathcal{M}_{d}$ be the space of all degree $d$ polynomials over $\C$ with a marked critical point, modulo affine conjugation. The space $\mathcal{M}_{d}$ can be seen as $(d-1)$-dimensional affine variety. Thurston's work shows that by fixing preperiodicities for \emph{all} of the critical points of a general degree $d$ polynomial map, one gets finitely many points in $\mathcal{M}_{d}$.  Instead of \emph{all} the critical points, fixing the behaviour of \emph{some} of the critical points results in dynamically interesting subfamilies of $\mathcal{M}_{d}$. These subfamilies have been studied since the 1990s (\cite{article2}, \cite{M90},  \cite{BDK91}, \cite{BH92},  \cite{MP92}, \cite{R06},  \cite{LQ07}, \cite{inbook}, \cite{BKM10}, \cite{article1}, \cite{AK20}). A long-standing conjecture of Milnor (\cite{inbook}) states that:
\begin{conjecture}\label{conj}
Fix $k\geq 0,n>0, d>2$. The family $\Sigma_{k,n,d}$ of conjugacy classes of degree $d$ polynomials for which the marked critical point is strictly $(k,n)$-preperiodic is an irreducible subset of $\mathcal{M}_{d}$, under the Zariski topology.
\end{conjecture}

Fixing the ramification index of the marked critical point to be equal to $d$, we arrive at the unicritical case. In other words, polynomials with a unique critical point in $\C$ are called \emph{unicritical polynomials}. Any unicritical polynomial of degree $d\geq 2$ is conjugate to a polynomial of the form $f_{c}(X)= X^{d} + c, c\in \C$ ($0$ is the unique finite critical point of $f_{c}$ with ramification index $d$).  In the article \cite{Milnor+2014+15+24}, Milnor studied arithmetic properties of unicritical polynomials and made the following conjecture:
\begin{conjecture}\label{uniconj}
    Fix $k \geq 0,n > 0,d \geq 2, d$ prime. The values of $c$ for which $0$ is strictly $(k,n)$-preperiodic under $f_{c}$, form one Galois orbit under the action of the absolute Galois group of $\Q$.   
 \end{conjecture}
 MIlnor's original question extends for non-prime degree polynomials as well, but there are more complications (see \cite[Remark~3.5]{Milnor+2014+15+24}). As study of unicritical polynomials in the moduli spaces of non-prime degree polynomials is out of the scope of this work (see Section~\ref{obs2}), we restrict ourself to the study of Conjecture \ref{uniconj}.
Conjecture \ref{uniconj} and its variants has seen a lot of work in the last decade (\cite{HT15}, \cite{FG15}, \cite{article1}, \cite{10.7169/facm/1799}, \cite{G20}, \cite{H21}, \cite{Han21}, \cite{BGok22}, \cite{BG23}). In particular, Vefa Goksel (\cite{10.7169/facm/1799}) has proved Conjecture \ref{uniconj} for prime degree polynomials with $k \geq 0, n=1$, by studying arithmetic properties of the forward orbit of $0$ under $f_{c}$. This result of Vefa Goksel has been used by Buff, Epstein and Koch (\cite{article1}) to prove Conjecture \ref{conj} for cubic polynomials ($d=3$), with period $n=1$. In a starkly different approach, Arfeux and Kiwi (\cite{AK20}) have proved the Conjecture \ref{conj} in the strictly periodic case, $k=0$, in the moduli space of cubic polynomials.

    As discussed above, Buff, Epstein and Koch proved Conjecture~\ref{conj} for prefixed curves in the moduli space of cubic polynomials using Goksel's result on Conjecture~\ref{uniconj}. This was the first major breakthrough in this conjecture. However, their proof can not be generalized for degree $\geq 5$. The problem to this generalization is two-fold. Firstly, the Irreducibility Lemma (\cite[Lemma~5]{article1}) does not apply to $\Sigma_{k,n,d}$ where the degree is $\geq 5$, as the origin is not a smooth point (Lemma~\ref{Eisen2}). Secondly, \cite[Lemma~12]{article1} can not be suitably generalized for higher degrees. Our approach to this problem resembles the approach of \cite{article1}. However, our techniques are applicable over moduli spaces of polynomial of arbitrary prime degree. One of the key observations in this regard is that Zariski closure of $\Sigma_{k,1,p}, k \geq 0, p$ prime, is the zero set of a polynomial which is of Eisenstein type. This is entirely new to the literature. Thanks to this result, we are able to prove Conjecture~\ref{conj} for $\Sigma_{k,1,p}, k \geq 0, p$ prime without relying on Conjecture~\ref{uniconj}, in contrast to the approach in \cite{article1} for the particular case of $p=3$. We also generalize the Irreducibility Lemma (\cite[Lemma~5]{article1}) for non-smooth points on a variety. Apart from these, the general case of moduli spaces of polynomials of arbitrary prime degree $p$, presents an entire array of new obstacles many of which become trivial for small values of $p$ (compare \cite{article1}).        
    
    We study Conjecture \ref{conj} for general degree $d$. Without any reliance on Conjecture \ref{uniconj}, we prove Conjecture \ref{conj} for prefixed families in the moduli space of odd prime degree polynomials. Moreover, we show that these quasi-affine varieties are of \emph{Eisenstein type}. This Eisenstein nature leads us to show irreducibility of subvarieties that arise from fixing the ramification index of the marked critical point. This also allows us to rediscover an irreducibility result of Vefa Goksel (\cite{10.7169/facm/1799}) in the unicritical case.

We show that (Theorem \ref{irrc}):
\begin{theorem}\label{theorem1}
    Fix $n=1$. Conjecture \ref{conj} is true for any odd prime degree $d$.
\end{theorem}

We use Theorem \ref{theorem1} to give a new and short proof of conjecture \ref{uniconj} for odd prime degree polynomials. In particular, we prove that (Theorem \ref{uni}):
\begin{theorem}\label{theorem2}
    Fix $n=1$. Conjecture \ref{uniconj} is true for any odd prime degree $d$.
\end{theorem}
\subsection{Mixed critical cases:} 
In Conjecture \ref{conj}, the ramification index of the marked critical point has been allowed to vary freely between $2$ and $d$. Fixing a lower bound $j$ for the ramification index and varying $j$ between $2$ and $d$, one obtains a finite decreasing sequence of quasi-affine varieties, with strictly decreasing dimensions. We call these sets mixed critical cases. The advantage of studying $\Sigma_{k,n,d}$ without any dependence on the unicritical case, is that it allows us to prove irreducibility of mixed critical subfamilies of $\mathcal{M}_{d}$. The mixed critical cases have not been studied much so far.  In \cite{R06}, Roesch studied the topology of the connectedness locus of a particular mixed critical family of maps. Here,  we consider algebraic properties of mixed critical subfamilies of $\Sigma_{k,1,p}$ and show that these are irreducible under Zariski topology on $\mathcal{M}_{p}$.
\begin{theorem}\label{theorem3}
    Fix $k\geq 0, p \text{ prime },p>3,\; 2<j<p$. The family of conjugacy classes of degree $p$ polynomials for which the marked critical point is strictly $(k,1)$-preperiodic and has ramification index $\geq j$, is an irreducible quasi-affine subvariety of $\mathcal{M}_{p}$.  
\end{theorem}
Theorem \ref{theorem3} has been restated and proved under Theorem \ref{inbetween}, in section \ref{mixcritical}. As a corollary of Theorem \ref{theorem3}, we get (Corollary \ref{mixedcoro}),
\begin{corollary}
    Let $k\geq 0, p$ be an odd prime and $2\leq j \leq p-1$. The set of all points of $\mathcal{M}_{p}$ for which the marked critical point is strictly $(k,1)$-preperiodic and has ramification index $j$, is an irreducible quasi-affine variety of dimension $p-j$.
\end{corollary}
Theorem \ref{theorem3} provides evidence for the following conjecture (Conjecture \ref{mixedconj1}):
\begin{conjecture}\label{mixconj}
    Fix $k\geq 0,\;n>0,\; d>2,\; 2 < j < d$. The family of conjugacy classes of degree $d$ polynomials for which the marked critical point is strictly $(k,n)$-preperiodic, and has ramification index $\geq j$, is an irreducible quasi-affine subvariety of $\mathcal{M}_{d}$.
\end{conjecture}
Observe that, allowing $j$ to be $2$ in Conjecture \ref{mixconj}, we obtain Conjecture \ref{conj}. Similarly in Conjecture \ref{mixconj}, allowing $j$ to be $d$ and restricting the irreducibility to only over $\Q$, we arrive at Conjecture \ref{uniconj}. Thus, Conjecture \ref{mixconj}, along with Theorem \ref{theorem3}, provides a bridge between Conjectures \ref{conj} and \ref{uniconj}.

We continue this study of irreducibility for $\Sigma_{k,2,3}, k \geq 0$ curves in a follow up article 

\subsection{Sketch of Proof.} \subsubsection{Normal form.} To study the families mentioned in Theorem \ref{theorem1}, we first require a concrete description of $\mathcal{M}_{d}$ as an affine variety. By Thurston's work, the dynamical behaviour of any polynomial is dictated by the dynamical behaviour of its critical points. This suggests representing polynomials by their finite critical points and one finite critical value, modulo affine conjugation. We introduce a normal form for monic, reduced degree $d\geq 2$ polynomials, which is a modification of \emph{Branner-Hubbard normal form} appearing in (\cite{article2}). This normal form also generalizes the unicritical normal form $f(X)= X^{d}+c$  used in (\cite{Milnor+2014+15+24}, \cite{HT15}, \cite{10.7169/facm/1799}, \cite{G20}, \cite{BGok22}, \cite{BG23}) and each of the various other normal forms considered in (\cite{article2}, \cite{inbook}, \cite{LQ07}, \cite{FG15}). 

This has two applications. First, we use this normal form to express $\mathcal{M}_{d}$ explicitly as an affine variety, $(d-1)$-dimensional affine space quotiented by the action of the group of $(d-1)$th roots of unity, and a permutation group. Secondly, we use the normal form to construct polynomials $h_{k,n,d},~ k\geq 0, n>0, d>2,$ such that the families $\Sigma_{k,n,d}$ appearing in Conjecture \ref{conj} are Zariski dense in the algebraic sets of these polynomials in $\mathcal{M}_{d}$.

\subsubsection{Proof of Theorem \ref{theorem1}: Irreducibility of $\Sigma_{k,1,p}$.} Let $k\geq 0, p$ be an odd prime. To show the irreducibility of $\Sigma_{k,1,p}$, it suffices to show that the polynomials $h_{k,1,p}$ are irreducible over $\C$. 

The proof of irreducibility uses arithmetical methods. We first observe that the coefficients of these polynomials lie in the ring $\Z_{(p)}$, the localization of the rational integers $\Z$ at the prime $p$. We show that the polynomials $h_{k,1,p}$ are generalised Eisenstein-type polynomials with respect to the prime $p$. For the resultant condition of the generalised Eisenstein criterion, we need a weaker version of Thurston's rigidity, of which an elementary algebraic proof is given in Subsection \ref{lemmasthurs}. This shows that $h_{k,1,p}$ polynomials are irreducible over $\Q$.

To prove irreducibility over the complex numbers, we observe that the lowest degree homogeneous part of $h_{k,1,p}$ factorizes into linear factors over $\Q$, for any $k\geq 2,\; p$ an odd prime. As $h_{k,1,p}$ polynomials are irreducible over $\Q$, this forces their irreducibility over $\C$, for any $k\geq 2$. The case $k=0$ is a trivial case, as $h_{0,1,p}$ is linear. For $k=1$, a little modification of the above proof shows that $h_{1,1,p}$ is also irreducible over $\C$. 

\subsubsection{Proof of Theorem \ref{theorem3}: Mixed critical case.} The variables of the polynomials $h_{k,1,p}$ stand for the critical points and one critical value. The polynomials defining the mixed critical locus arise by identifying some of the variables (critical points). The resulting polynomials are once again generalised Eisenstein polynomials, and hence are irreducible over $\Q$. As long as we do not identify all the critical points,  the lowest degree homogeneous part of the resulting polynomials mentioned above always has at least one simple linear factor defined over $\Q$. This enforces their irreducibility over $\C$, completing the proof of irreducibility of all the mixed critical loci. 

\subsubsection{Proof of Theorem \ref{theorem2}: Unicritical case.} Identifying all the finite critical points to $0$, we reach at the unicritical case. The resulting polynomial we get from $h_{k,1,p}$ is a univariate polynomial, which is an Eisenstein-type polynomial multiplied with some power of the variable $\beta$ (critical value). Throwing away this power of $\beta$, we see that the roots of these Eisentstein-type polynomials are the values of $c$ mentioned in Conjecture \ref{uniconj}. This allows us to complete the proof of Theorem \ref{theorem2}.

\subsection{Outline of the paper.} Section \ref{prelim} contains construction of normal form, moduli space $\mathcal{M}_{d}$ and the polynomials $h_{k,n,d}$. All the notations are collected in Section \ref{notations}. We gather some basic lemmas and tools, such as generalised Eisenstein criterion, in Section \ref{lemmas}. A short algebraic proof of a weaker version of Thurston's rigidity theorem is also included in Section \ref{lemmas}. Proof of Theorem \ref{theorem1} is the content of Section \ref{t'''}. In Sections \ref{mixcritical} and \ref{unicritical}, we prove Theorem \ref{theorem3} (mixed critical case), and Theorem \ref{theorem2} (unicritical case), respectively. Obstructions to direct generalization of our techniques to some other cases of Conjecture \ref{conj} are studied in Sections \ref{obs1} and \ref{obs2}.

\subsection{Acknowledgement.} The author would like to thank C. S. Rajan, Sagar Shrivastava and Manodeep Raha, for their valuable inputs and many discussions. The author expresses his gratitude to Ashoka University for providing a productive workspace, where this work was done. The author is also thankful to Kerala School of Mathematics for his visit, where he had many helpful discussions with Plawan Das, Subham Sarkar and M. M. Radhika.

\section{Preliminaries}\label{prelim}
Let $f$ be a polynomial over $\C$. Let $f^{m}, m\in \N$, denote the iteration $f^{m}:=f\circ f^{m-1}$, and $f^{0}$ is defined to be the identity map. For any point $x\in \C$, the \emph{forward orbit} of $x$ is defined to be the set, $\{x,f(x),f^{2}(x),\cdot \cdot \cdot,f^{m}(x),\cdot \cdot \cdot\}=\{f^{m}(x)|m\in \N\cup \{0\}\}$. Recall that the roots of the derivative $f'$ of $f$, are called the \emph{finite critical points} or \emph{ramified points} of $f$. Since the time of Fatou and Julia, it has been known that behaviour of a complex dynamical system (which means forward orbits of $x$, where $x$ varies over $\C$) is largely governed by the forward orbits of the critical points of the map. Hence, in order to understand the dynamical system generated by $f$, one needs to study the forward orbits of critical points of $f$.
\subsection{Periodic and preperiodic points} Let $g$ be a polynomial defined over $\C$. Let $k \geq 0,n >0$ be integers. A point $x\in \C$ is called a \emph{periodic point of period $n$} (or, \emph{$n$-periodic point}) if and only if $g^{n}(x)=x$, and called \emph{strictly $n$-periodic point} (or, \emph{periodic point of exact period} $n$) if and only if $n$ is the smallest natural number such that $g^{n}(x)=x$.

 A point $x\in \C$ is called a $(k,n)$-\emph{preperiodic point} if and only if $g^{k+n}(x)=g^{k}(x)$, and \emph{strictly $(k,n)$-preperiodic point} if and only if $g^{k+n}(x)=g^{k}(x)$ and for all $0\leq l\leq k,\; 0\leq m\leq n$, $(l,m)\neq (k,n)\in \Z^{2} $, we have $g^{l+m}(x)\neq g^{l}(x)$.
\subsection{Affine conjugation} Consider $Aut(\C)=\{h\in \C[z]~|~h(z)=az+b, a\in \C\backslash \{0\}, b\in \C\}$, the group of holomorphic automorphisms of $\C$. Two maps $g,\tilde{g}\in \C[z]$ are called \emph{affine conjugate} to each other if and only if there exists a map $h\in Aut(\C)$, such that $\tilde{g}=h\circ g\circ h^{-1}$.
\subsection{Normal form and moduli spaces of polynomials with a marked critical point.} For an arbitrary but fixed integer $d>2$, let $S_{d}$ be the set of all polynomials of degree $d$ with a marked critical point. Consider the \emph{moduli space} $\mathcal{M}_{d}$, which is defined to be the quotient space of $S_{d}$, obtained by identifying polynomials that are affine conjugate to each other and the affine conjugation map sends the marked critical point of the first to the marked critical point of the latter. We say a polynomial is \emph{reduced} if the sum of its roots is zero, which is also equivalent to saying that the coefficient of the second highest degree term of the polynomial is zero. Any polynomial over $\C$ is affine conjugate to a monic reduced polynomial over $\C$. Thus, the moduli space $\mathcal{M}_{d}$ is the set of equivalence classes of monic reduced polynomials of degree $d$ over $\C$ with a marked critical point, where the equivalence relation is given by affine conjugation.

Let $f$ be a monic, reduced, degree $d$ polynomial. We define \emph{normal form} of $f$ to be a form of the polynomial $f$ in which each coefficient of $f$ is expressed as polynomials in its finite critical points and one finite critical value. Note that, any polynomial of degree $d$ has $(d-1)$ number of finite critical points, counted with multiplicity. 
\begin{proposition}\label{propnormal}
    Let $f$ be a monic and reduced element of $\C[z]$ of degree $d$. Let $\alpha_{1},\alpha_{2},\cdot \cdot \cdot,\alpha_{d-1}$ be the finite critical points of $f$, not necessarily distinct. Let $f(\alpha_{1})= \beta,\; \bar{\alpha}:= (\alpha_{1},\alpha_{2},\cdot \cdot \cdot,\alpha_{d-1})$ and $s_{i}(\bar{\alpha})$ denote the $i$-th elementary symmetric polynomial in $\alpha_{1},\cdot \cdot \cdot,\alpha_{d-1}$. Then, the polynomial $f(z)$ can be written in a normal form as,
    $$ z^{d} +  \sum_{i=2}^{d-1} (-1)^{i}\frac{d}{d-i}\cdot s_{i}(\bar{\alpha})\cdot z^{d-i} - \alpha_{1}^{d} - \sum_{i=2}^{d-1}(-1)^{i}\frac{d}{d-i}\cdot s_{i}(\bar{\alpha})\cdot \alpha_{1}^{d-i} + \beta,$$
    with $\sum_{i=1}^{d-1} \alpha_{i}=0$.
\end{proposition}
\begin{proof}
The derivative of $f$ can be written as,
\begin{center}
    $f'(z)=c(z-\alpha_{1})(z-\alpha_{2})\cdot \cdot \cdot(z-\alpha_{d-1}),$
\end{center}
where $c\in \C,c$ constant.
Since $f$ is reduced, $f'$ is also reduced. Hence, $\sum_{i=1}^{d-1} \alpha_{i} = 0$. So,
\begin{equation}\label{a}
     f'(z)=c\left( z^{d-1} + \sum_{i=2}^{d-1} (-1)^{i} s_{i}(\bar{\alpha})z^{d-i-1} \right),
\end{equation}
 where $\bar{\alpha}=(\alpha_{1},\alpha_{2},\cdot \cdot \cdot,\alpha_{d-1})$ and $s_{i}(\bar{\alpha})$ is the $i$-th elementary symmetric polynomial in $\alpha_{1},\alpha_{2},\cdot \cdot \cdot,\alpha_{d-1}$. By integration,
\begin{equation*}
    f(z)= c\left( \frac{z^{d}}{d} + \sum_{i=2}^{d-1}(-1)^{i}\frac{s_{i}(\bar{\alpha})}{d-i}z^{d-i} \right) +c',
\end{equation*}
where $c'\in \C$, $c'$ constant. The polynomial $f$ being monic, we get $c=d$. Hence,
\begin{equation*}
    f(z)= z^{d} + \sum_{i=2}^{d-1}(-1)^{i}\frac{d}{d-i}\cdot s_{i}(\bar{\alpha})\cdot z^{d-i} +c'.
\end{equation*}
Equating $f(\alpha_{1})$ with $\beta$, 
 \begin{equation*}
     \beta=f(\alpha_{1})= \alpha_{1}^{d} + \sum_{i=2}^{d-1}(-1)^{i}\frac{d}{d-i}\cdot s_{i}(\bar{\alpha})\cdot \alpha_{1}^{d-i} +c',
 \end{equation*}
 \begin{equation*}
     \implies c'= - \alpha_{1}^{d} - \sum_{i=2}^{d-1}(-1)^{i}\frac{d}{d-i}\cdot s_{i}(\bar{\alpha})\cdot \alpha_{1}^{d-i} + \beta.
 \end{equation*}
Hence,
\begin{equation}\label{b}
    f(z)= z^{d} + \sum_{i=2}^{d-1}(-1)^{i}\frac{d}{d-i}\cdot s_{i}(\bar{\alpha})\cdot z^{d-i} - \alpha_{1}^{d} - \sum_{i=2}^{d-1}(-1)^{i}\frac{d}{d-i}\cdot s_{i}(\bar{\alpha})\cdot \alpha_{1}^{d-i} + \beta.
\end{equation}
\end{proof}
In this article, we will represent each point of $\mathcal{M}_{d}$ by a monic reduced polynomial written in the normal form as in Equation \eqref{b}, with $\alpha_{1}$ as its marked critical point.
\begin{remark}
    The normal form mentioned above is a modification of the \emph{Branner-Hubbard normal form} appearing in \cite{article2}. This also generalises various other normal forms considered in \cite{inbook}, \cite{LQ07}, \cite{FG15}, \cite{Milnor+2014+15+24}.
\end{remark}
\subsection{Moduli space $\mathcal{M}_{d}$ is an affine variety}
As $f$ varies over all monic reduced polynomials over $\C$ of degree $d$, $\alpha_{i}, 1\leq i \leq p-1,$ and $\beta$ vary over $\C$, such that $\sum_{i=1}^{d-1} \alpha_{i}=0$. So, we can parametrize the set of all monic, reduced polynomials over $\C$ of degree $d$ with a marked critical point, using $\alpha_{1}, \alpha_{2},\cdot \cdot \cdot,\alpha_{d-2}$ as $(d-2)$ finite critical points with $\alpha_{1}$ as its marked critical point and $\beta$ such that $f(\alpha_{1})=\beta$. From here on in this article, we will consider $\alpha_{1},\alpha_{2},\cdot \cdot \cdot,\alpha_{d-2}, \beta$ as variables over $\C$, unless otherwise mentioned. For the rest of this subsection, we will denote the polynomial $f$ in Equation \eqref{b} as $f_{\bar{\alpha},\beta}$, where $\bar{\alpha}=(\alpha_{1},\alpha_{2},\cdot \cdot \cdot,\alpha_{d-1})$. 

Following the previous subsection and the above paragraph, $\mathcal{M}_{d}$ can be considered as a quotient space of $\C^{d-1}$, quotiented by affine conjugation relation, where the coordinates of $\C^{d-1}$ are $\alpha_{1},\alpha_{2},\cdot \cdot \cdot,\alpha_{d-2},\beta$. Next, we will show that $\mathcal{M}_{d}$ is an affine variety.
\begin{proposition}\label{affvar}
    For any $d\in \N,\; d > 2$, the space $\mathcal{M}_{d}$ is an affine variety.
\end{proposition}
\begin{proof}
Consider the affine space $\C^{d}$ with $\alpha_{1},\cdot \cdot \cdot,\alpha_{d-1},\beta$ as coordinates. Consider the hyperplane $V$ which is the zero set of $\sum_{i=1}^{d-1} \alpha_{i}$ in $\C^{d}$. We have an affine isomorphism,
\begin{equation}\label{firstmap}
    \C^{d-1} \rightarrow V,\;\;\; (\alpha_{1},\cdot \cdot \cdot,\alpha_{d-2},\beta) \mapsto (\alpha_{1},\cdot \cdot \cdot,\alpha_{d-2},-\sum_{i=1}^{d-2} \alpha_{i}, \beta).
\end{equation}
As we have seen above, we have a surjective map,
\begin{equation}\label{secondmap}
    V \twoheadrightarrow \mathcal{M}_{d},\;\;\; \; (\alpha_{1},\alpha_{2},\cdot \cdot \cdot,\alpha_{d-2},\alpha_{d-1},\beta) \mapsto f_{\bar{\alpha},\beta}.
\end{equation}
Now, two points $(\alpha_{1},\alpha_{2},\cdot \cdot \cdot,\alpha_{d-2},\alpha_{d-1},\beta)$ and $(\alpha'_{1},\alpha'_{2},\cdot \cdot \cdot,\alpha'_{d-2},\alpha'_{d-1},\beta')$ in $V$ maps to the same point in $\mathcal{M}_{d}$ iff there is an affine conjugation map $h$ such that $h\circ f_{\bar{\alpha},\beta}\circ h^{-1} = f_{\bar{\alpha'}, \beta'}$ and $h(\alpha_{1})=\alpha'_{1}$.

If $h$ is the identity map, then $f_{\bar{\alpha},\beta}=f_{\bar{\alpha'},\beta'}, \alpha_{1}=\alpha'_{1}$ and $\beta=\beta'$. Hence, the ordered set $\{\alpha'_{2},\cdot \cdot \cdot,\alpha'_{d-2},\alpha'_{d-1}\}$ is a permutation of $\{\alpha_{2},\cdot \cdot \cdot,\alpha_{d-2},\alpha_{d-1}\}$, and every such permutation will map to the same point in $\mathcal{M}_{d}$.

Let $h$ be a non-identity map. As $f_{\bar{\alpha},\beta},f_{\bar{\alpha'},\beta'}$ are both reduced, monic and of degree $d$, the affine conjugation map between them has to be $h(z)=\zeta \cdot z$, for some $(d-1)$th root of unity $\zeta$. A brief computation shows that $h^{-1}\circ f_{\bar{\alpha},\beta} \circ h = f_{\zeta^{-1}\cdot \bar{\alpha},\zeta^{-1}\cdot \beta}$.  

Consider the following two finite subgroups of $GL_{d}(\C)$, 
\begin{equation*}
    \mu_{d-1}:= \{\zeta I\;|\; I \text{ is the identity matrix}, \zeta \text{ is a } (d-1)\text{th root of unity}\}.
\end{equation*}
\begin{equation*}
    P:= \text{Group of all } d\times d \text{ permutation matrices which fix the first and last coordinate.}
\end{equation*}
The hyperplane $V$ is closed under the action of both $P$ and $\mu_{d-1}$. From Equation \eqref{secondmap} and last four paragraphs, we get that there is a bijective map, 
\begin{equation}\label{thirdmap}
    V/\langle P,\mu_{d-1} \rangle\;\;\; \longleftrightarrow \;\;\; \mathcal{M}_{d}.
\end{equation}
As $\mu_{d-1}$ and $P$ commute with each other, the group $\langle P,\mu_{d-1} \rangle$ is a finite group. So, the quotient space $V/\langle P,\mu_{d-1} \rangle$ is an affine variety. Putting Equations \eqref{firstmap}, \eqref{thirdmap} together,
\begin{equation}\label{finalmap}
    \C^{d-1}\;\; \xrightarrow{\simeq} \;\; V \;\; \twoheadrightarrow \;\; V/\langle P,\mu_{d-1} \rangle\;\; \longleftrightarrow \;\;\mathcal{M}_{d},
\end{equation}
we see that $\mathcal{M}_{d}$ is an affine variety.
\end{proof}

\subsection{The reduced polynomial condition: $\sum_{i=1}^{d-1} \alpha_{i} = 0$}
The polynomial $f$ in Equation \eqref{b} lies in $\Q[z,\alpha_{1},\alpha_{2},\cdot \cdot \cdot,\alpha_{d-1},\beta]$. In the later sections, we would want to apply the relation $\sum_{i=1}^{d-1} \alpha_{i} = 0$ on the polynomial $f$ and the multivariate polynomials we form using $f$. Here we will study a ring homomorphism between multivariate polynomial rings whose kernel is the ideal generated by $\sum_{i=1}^{d-1} \alpha_{i}$.

Fix an integer $d>2$. Let $R$ be a subring of $\C$ containing $1$. Consider the ring homomorphism,
\begin{equation} \label{f'}
    \Psi_{d,R}: R[z,\alpha_{1},\alpha_{2},\cdot \cdot \cdot,\alpha_{d-1},\beta] \longrightarrow R[z,\alpha_{1},\alpha_{2},\cdot \cdot \cdot,\alpha_{d-2},\beta],
\end{equation}
which is identity map on $R[z,\alpha_{1},\alpha_{2},\cdot \cdot \cdot,\alpha_{d-2},\beta]$ and maps $\alpha_{d-1}$ to $-(\sum_{i=1}^{d-2} \alpha_{i})$. Kernel of this map is the ideal generated by $\sum_{i=1}^{d-1} \alpha_{i}$ in $R[z,\alpha_{1},\alpha_{2},\cdot \cdot \cdot,\alpha_{d-1},\beta]$. For any $f \in R[z,\alpha_{1},\alpha_{2},\cdot \cdot \cdot,\alpha_{d-1},\beta]$, we denote the element $\Psi_{d,R}(f)$ as $\hat{f}$.

For two polynomials $g,h\in R[z,\alpha_{1},\alpha_{2},\cdot \cdot \cdot,\alpha_{d-1},\beta]$, we define $\circ$ to be the composition with respect to $z$, i.e. $g\circ h:= g(h(z,\alpha_{1},\cdot \cdot \cdot,\alpha_{d-1},\beta), \alpha_{1},\cdot \cdot \cdot,\alpha_{d-1},\beta)$ and $\hat{g} \circ \hat{h}:= \hat{g}(\hat{h}(z,\alpha_{1},\cdot \cdot \cdot,\alpha_{d-2},\beta),\alpha_{1},\cdot \cdot \cdot,\alpha_{d-2},\beta)$. Note that, $\Psi_{d,R}$ being a ring homomorphism,
\begin{equation}\label{hatcommcomp}
    \widehat{g \circ h}= \hat{g} \circ \hat{h}.
\end{equation}

\subsection{The polynomials $h_{k,n,d}$}
In this subsection, we will form a polynomial $h_{k,n,d}$ in $\Q[\alpha_{1},\alpha_{2},\cdot \cdot \cdot,\alpha_{d-2},\beta]$. Then we will show that to prove irreducibility of $\Sigma_{k,n,d}$ in $\mathcal{M}_{d}$, it is enough to show that $h_{k,n,d}$ is irreducible over $\C$.

From Equation \eqref{b}, let 
\begin{equation*}
    f(z)= z^{d} + \sum_{i=2}^{d-1}(-1)^{i}\frac{d}{d-i}\cdot s_{i}(\bar{\alpha})\cdot z^{d-i} - \alpha_{1}^{d} - \sum_{i=2}^{d-1}(-1)^{i}\frac{d}{d-i}\cdot s_{i}(\bar{\alpha})\cdot \alpha_{1}^{d-i} + \beta,
\end{equation*}
 with $\sum_{i=1}^{d-1} \alpha_{i} = 0$. Let us assume that $\alpha_{1}$ is strictly $(k,n)$-preperiodic under $f$. Then, we get $f^{k+n}(\alpha_{1})-f^{k}(\alpha_{1})=0$. From the definition of $f$, the polynomial
\begin{equation*}
    f_{k,n,d}:= f^{k+n}(\alpha_{1})-f^{k}(\alpha_{1}),
\end{equation*}
 lies in $\Q[\alpha_{1},\alpha_{2},\cdot \cdot \cdot,\alpha_{d-1},\beta]$. Consider $ \widehat{f_{k,n,d}} $ and $ \hat{f}_{k,n,d} := (\hat{f})^{k+n}(\alpha_{1}) - (\hat{f})^{k}(\alpha_{1})$, where $\; \hat{} \;$ is as defined in the previous subsection. From Equation \eqref{hatcommcomp}, we get
 \begin{equation}\label{fhatcommcomp}
     \widehat{f_{k,n,d}} =  \hat{f}_{k,n,d}.
 \end{equation}
 The set of solutions of the polynomial $\hat{f}_{k,n,d}$ in $\C^{d-1}$ contains all the points in $\C^{d-1}$ for which $\alpha_{1}$ is (not necessarily strictly) $(k,n)$-preperiodic under $\hat{f}$. Consider $\hat{f}_{l,m,d}$, defined similarly as $\hat{f}_{k,n,d}$, for any non-negative integers $l,m$ such that $l\leq k, m|n, (l,m)\neq (k,n) \in \Z^{2}$. We will show in Lemma \ref{l} that for any such $(l,m)$, the polynomial $\hat{f}_{l,m,d}$ divides $\hat{f}_{k,n,d}$ in $\Q[\alpha_{1},\alpha_{2},\cdot \cdot \cdot,\alpha_{d-2},\beta]$.

Let,
\begin{equation}\label{g'}
    h_{k,n,d}:= \frac{\hat{f}_{k,n,d}}{\prod_{i} g_{i}^{a_{i}}} =  \frac{\widehat{f_{k,n,d}}}{\prod_{i} g_{i}^{a_{i}}},
\end{equation}
where $g_{i}$ varies over all irreducible factors of $\hat{f}_{l,m,d}= \widehat{f_{l,m,d}}$ in $\Q[\alpha_{1},\alpha_{2},\cdot \cdot \cdot,\alpha_{d-2},\beta]$, for all $0\leq l \leq k, 1\leq m \leq n, m|n, (l,m) \neq (k,n) \in \Z^{2}$, and $a_{i}$ is the highest power of $g_{i}$ that divides $\widehat{f_{k,n,d}}$. 

\begin{lemma}\label{h}
Let $k\in \N \cup \{0\},n,d\in \N, d>2$. Let us assume that $h_{k,n,d}$ is irreducible over $\C$. Then, the set of points of the form $(\alpha_{1},\alpha_{2},\cdot \cdot \cdot,\alpha_{d-2},\beta)$ in $\C^{d-1}$ for which $\alpha_{1}$ is strictly $(k,n)$-preperiodic under $\hat{f}$, is open and Zariski dense in the algebraic set of $h_{k,n,d}$. 
\end{lemma}
\begin{proof}
Let us denote the algebraic set of $h_{k,n,d}$ in $\C^{d-1}$ as $V(h_{k,n,d})$. Any point $(\alpha_{1},\alpha_{2},\cdot \cdot \cdot,\alpha_{d-2},\beta) \in \C^{d-1}$ for which $\alpha_{1}$ is strictly $(k,n)$-preperiodic under $\hat{f}$, is a solution to the equation $\hat{f}_{k,n,d}$ but not a solution of any of the $g_{i}$'s in Equation \eqref{g'}. Hence, all such points lie in $V(h_{k,n,d})$. Set of all these points are precisely the set $U:=V(h_{k,n,d})\setminus \cup_{g_{i}} V(h_{k,n,d},g_{i})$, where $g_{i}$ varies over all irreducible factors of $\hat{f}_{l,m,d},$ for all $0\leq l\leq k, 1\leq m\leq n, m|n, (l,m)\neq (k,n)\in \Z^{2}$. By definition of $h_{k,n,d}$, the polynomial $h_{k,n,d}$ is coprime to $g_{i}$ over $\Q$, for every $g_{i}$. Hence, they are coprime over $\C$ too. So, $U$ is open and non-empty. As $h_{k,n,d}$ is irreducible, $U$ is Zariski dense in $V(h_{k,n,d})$.
\end{proof}

\begin{remark}
    Observe that, the proof of Lemma~\ref{h} shows that $\Sigma_{k,n,d}$ is non-empty for arbitrary $k,n,d$.
\end{remark}

Let us define $\Sigma_{k,n,d}$ to be the set of all points in $\mathcal{M}_{d}$ for which the marked critical point is strictly $(k,n)$-preperiodic. From Lemma \ref{h}, we get the following corollary, 
\begin{corollary} \label{coro1}
    If $h_{k,n,d}$ is irreducible over $\C$, then $\Sigma_{k,n,d}$ is an irreducible quasi-affine subvariety of $\mathcal{M}_{d}$.\hspace{240pt} \qedsymbol
\end{corollary}

\begin{remark}
Our first goal in this article is to prove the irreducibility of $\Sigma_{k,1,p}, k\geq 0, p$ odd prime. Let us summarize our method. Consider the diagram below. Let $\C[d]:= \C[\alpha_{1}, \alpha_{2},\cdot \cdot \cdot, \alpha_{d}, \beta],$ for any $d\in \N$. The polynomial $f_{k,n,d} \in \C[d-1]$ kills all $(\alpha_{1}, \alpha_{2}, \cdot \cdot \cdot, \alpha_{d-1}, \beta)$ for which $\alpha_{1}$ is (not necessarily strictly) $(k,n)$-preperiodic under $f$ as in Equation \eqref{b}. We take the image of $f_{k,n,d}$ under $\; \hat{} \;$ map. From $\hat{f}_{k,n,d} \in \C[d-2]$, we factor out all common irreducible factors of $\hat{f}_{k,n,d}$ and $\hat{f}_{l,m,d}$, $0\leq l\leq k, 1\leq m\leq n, m|n, (l,m)\neq (k,n) \in \Z^{2}$, with each factor raised to the highest power that divides $\hat{f}_{k,n,d}$. We call the resulting polynomial $h_{k,n,d} \in \C[d-2]$. Assuming $h_{k,n,d}$ is irreducible, we get that $V(h_{k,n,d})\subseteq \C^{d-1}$ is irreducible. The algebraic set $V(h_{k,n,d})$ is irreducible implies the set $S$ of all points of $\C^{d-1}$ for which $\alpha_{1}$ is strictly $(k,n)$-preperiodic, is non-empty and open in $V(h_{k,n,d})$. Hence, $S$ is irreducible. So the image of $S$ under the horizontal sequence of arrows in the diagram is irreducible, which is $\Sigma_{k,n,d}$.

\begin{tikzcd}
   & \C[d-1] \arrow[ld, "\widehat{}"] \arrow[d, twoheadrightarrow] &  & \\
   \C[d-2] \arrow[d, dashed, leftrightarrow] & \sfrac{\C[d-1]}{(\sum_{i=1}^{d-1} \alpha_{i})} \arrow[l, "\simeq"] \arrow[d, dashed, leftrightarrow] &  & \\
   \C^{d-1} \arrow[r, "\simeq"] & V = V(\sum_{i=1}^{d-1} \alpha_{i}) \arrow[r, twoheadrightarrow] & V/\langle P,\mu_{d-1} \rangle \arrow[r, leftrightarrow] & \mathcal{M}_{d}\\
\end{tikzcd}

In Section \ref{t'''}, we will show that $h_{k,1,p}$ polynomials are irreducible over $\C$, for any $k\geq 0$ and odd prime $p$. By Corollary \ref{coro1}, this shows that $\Sigma_{k,1,p}, k\geq 0, p$ odd prime, is irreducible, completing the proof of Theorem \ref{theorem1} stated in the introduction.
\end{remark}

\section{Notations}\label{notations}
 We will use the following notations for the rest of the article. Let $R$ be a ring and $d\in \N$.
 \begin{itemize}
     \item $R[d]:= R[\alpha_{1}, \alpha_{2}, \cdot \cdot \cdot, \alpha_{d}, \beta],$ the multivariate polynomial ring over $R$ in variables $\alpha_{1}, \alpha_{2},\cdot \cdot \cdot,\alpha_{d}, \beta$.
 \end{itemize}
 Let $d>2$. Let $g,h$ be elements of $R[d-1].$
\begin{itemize}
    \item By $\hat{g}$, we denote the image of $g$ under the map $\Psi_{d}: R[d-1] \rightarrow R[d-2]$, that is the identity map on $R[\alpha_{1},\alpha_{2},\cdot \cdot \cdot,\alpha_{d-2},\beta]$ and maps $\alpha_{d-1}$ to $-(\sum_{i=1}^{d-2} \alpha_{i})$.
    \item By saying $g$ is monic in $R[\alpha_{1},\alpha_{2},\cdot \cdot \cdot,\alpha_{d-1}][\beta]$, we mean $g$ is monic as a polynomial in $\beta$ over the ring $R[\alpha_{1},\alpha_{2},\cdot \cdot \cdot,\alpha_{d-1}]$.
    \item By $\mbox{Res}(g,h)$, we denote the resultant of $g$ and $h$, both considered as polynomials in $\beta$ with coefficients coming from the integral domain $R[\alpha_{1},\alpha_{2},\cdot \cdot \cdot,\alpha_{d-1}]$. So, $\mbox{Res}(g,h)\in R[\alpha_{1},\alpha_{2},\cdot \cdot \cdot,\alpha_{d-1}]$.
\end{itemize}
Consider the polynomial $f$ as defined in Equation \eqref{b}. For any non-negative integers $k,n$, with $n>0$,  
\begin{itemize}
    \item $f^{0}:=$ identity map, $f^{1}:=f, f^{n}:=f^{n-1} \circ f, $ for all $ n\in \N$.
    \item $f'$ denote the derivative of $f$ w.r.t $z$.
    \item $f_{k,n,d}=f_{k,n,d}(\alpha_{1},\alpha_{2},\cdot \cdot \cdot,\alpha_{d-1},\beta):= f^{k+n}(\alpha_{1})-f^{k}(\alpha_{1})$.
    \item $\hat{f}_{k,n,d}= \widehat{f_{k,n,d}} := \Psi_{d,\C}(f_{k,n,d}),$ where $\Psi_{d,\C}$ is as defined in Equation \eqref{f'}. 
    \item $h_{k,n,d}=h_{k,n,d}(\alpha_{1},\alpha_{2},\cdot \cdot \cdot,\alpha_{d-1},\beta):= \hat{f}_{k,n,d}/\prod_{i} g_{i}^{a_{i}}$, where $g_{i}$'s vary over all distinct irreducible factors of $\hat{f}_{l,m,d}$, where $l\leq k, m|n, (l,m)\neq (k,n)\in \Z^{2}$, and for each $i$, $a_{i}$ is the highest power of $g_{i}$ that divides $\hat{f}_{k,n,d}$ over $\Q$.
    \item $\C^{n}:=$ complex affine space of dimension $n$.
    \item $\mathcal{M}_{d}:=$ the moduli space of degree $d$ polynomials with a marked critical point.
    \item $\Sigma_{k,n,d}:=$ the set of all points of $\mathcal{M}_{d}$ for which the marked critical point is strictly $(k,n)$-preperiodic.
    \item $\Phi_{n}(z,w):= \frac{z^{n}-w^{n}}{z-w}$, for all $n\in \N$.
    \item $G_{\Q}$ denote the absolute Galois group of $\Q$.
\end{itemize}
\section{Basic lemmas and Tools}\label{lemmas}
\subsection{Divisibility properties of $\hat{f}_{k,n,d}$}

Let $d$ be a natural number, $d>2$. In this subsection, we will study some divisibility properties of $\hat{f}_{k,n,d}$, for non-negative integers $k,n$, with $n>0$.
\begin{lemma}\label{l}
Let $k,l \in \N \cup \{0\},n,m \in \N$ such that $ l\leq k$ and $m|n$, the polynomial $f_{l,m,d}$ divides $f_{k,n,d}$ in $\Q[d-1]$. Hence, $\hat{f}_{l,m,d}$ divides $\hat{f}_{k,n,d}$ in $\Q[d-2]$.
\end{lemma}
\begin{proof}
To prove the lemma, it is enough to show that $f_{k-1,n,d}$ divides $f_{k,n,d}$ and $f_{k,m,d}$ divides $f_{k,n,d}$ in $\Q[d-1]$, for any $k,n,m\in \N,$ such that $m|n$.

Consider the polynomial $f$ in Equation \eqref{b} as a polynomial in a single variable, $z$, over the integral domain $\Q[d-1]$. Observe that, $(z-w)$ divides $f(z)-f(w)$. Putting $z=f^{k+n-1}(\alpha_{1})$, $w=f^{k-1}(\alpha_{1})$, one gets that $f_{k-1,n,d}$ divides $f_{k,n,d}$ in $\Q[d-1]$.

Hence, for any non-negative integers $l,k,n,$ such that $l\leq k, n>0$, the polynomial $f_{l,n,d}$ divides $f_{k,n,d}$ in $\Q[d-1]$. Let $n=mt, t\in \N$. Then,
\begin{align*}
    f_{k,n,d}& =f_{k,mt,d}\\
    & =f_{k+m(t-1),m,d} + f_{k+m(t-2),m,d} + \cdot \cdot \cdot + f_{k,m,d}\\
    & = \sum_{i=0}^{t-1} f_{k+mi,m,d}.
\end{align*}
As $f_{k,m,d}$ divides $f_{k+mi,m,d}, $ for $i \geq 0$, the polynomial $f_{k,m,d}$ divides $f_{k,n,d}$, whenever $m$ divides $n$. The lemma is proved.  
\end{proof}
\begin{lemma}\label{k}
 Let $g$ be an irreducible element of $\Q[d-2]$, monic as a polynomial in  $\Q[\alpha_{1},\alpha_{2},\cdot \cdot \cdot,\alpha_{d-2}][\beta]$. Let $k,l,m,n$ be non-negative integers with $m,n$ non-zero, $l\leq k, $ g.c.d $(m,n)=r$. If $g$ divides both $\hat{f}_{k,n,d}$ and $\hat{f}_{l,m,d}$ in $\Q[d-2]$, then $g$ divides $\widehat{f_{l,r,d}}$ in $\Q[d-2]$.
\end{lemma}
\begin{proof}
Let us consider the algebraic set $V(g)$ in $\C^{d-1}$. Let $(\alpha_{1}^{0},\alpha_{2}^{0},\cdot \cdot \cdot,\alpha_{d-2}^{0},\beta^{0})\in V(g)$. As $g$ divides both $\hat{f}_{k,n,d}$ and $\hat{f}_{l,m,d}$, $\alpha_{1}^{0}$ is both $(k,n)$-preperiodic and $(l,m)$-preperiodic, for any polynomial corresponding to the point $(\alpha_{1}^{0},\alpha_{2}^{0},\cdot \cdot \cdot,\alpha_{d-1}^{0},\beta^{0})$. So, $\alpha_{1}^{0}$ is $(l,r)$-preperiodic. Hence, $V(g)\subseteq V(\hat{f}_{l,r,d})$. So in $\C[d-2]$,
\begin{center}
    $rad(g)=I(V(g))\supseteq I(V(\hat{f}_{l,r,d}))= rad(\hat{f}_{l,r,d})\supseteq (\hat{f}_{l,r,d})$.
\end{center}
Now, $g$ being irreducible in $\Q[d-2]$, $g$ is seperable over $\C$. So, we have $(g)=rad(g)\supseteq (\hat{f}_{l,r,d})$. So, $g$ divides $\hat{f}_{l,r,d}$ in $\C[d-2]$. Hence $g$ divides $\hat{f}_{l,r,d}$ in $\Q[d-2]$, as both $g$ and $\hat{f}_{l,r,d}$ are in $\Q[d-2]$.
\end{proof}
From Lemmas \ref{l} and \ref{k}, one directly obtains the following corollary,
\begin{corollary}
    Let $k,l,m,n$ be non-negative integers with $m,n$ non-zero, $l\leq k,$ g.c.d $(m,n)=r$. Then,  $\hat{f}_{l,r,d}$ divides g.c.d $(\hat{f}_{k,n,d},\hat{f}_{l,m,d})$ in $\Q[d-2]$. Moreover, the radical ideals of the ideal generated by $\hat{f}_{l,r,d}$ and the ideal generated by g.c.d $(\hat{f}_{k,n,d},\hat{f}_{l,m,d})$ in $\C[d-2]$ are same. \hspace{190pt} \qedsymbol
\end{corollary}
From here on, we will consider $d$ to be an odd prime power. Let $p$ be an odd prime and $e\in \N$, both chosen arbitrarily but fixed. Consider the polynomial expression in Equation $\eqref{b}$ with $d=p^{e}$, 
\begin{equation}\label{i}
    f(z)= z^{p^{e}} + \sum_{i=2}^{p^{e}-1}(-1)^{i}\frac{p^{e}}{p^{e}-i}\cdot s_{i}(\bar{\alpha})\cdot z^{p^{e}-i} - \alpha_{1}^{p^{e}} - \sum_{i=2}^{p^{e}-1}(-1)^{i}\frac{p^{e}}{p^{e}-i}\cdot s_{i}(\bar{\alpha})\cdot \alpha_{1}^{p^{e}-i} + \beta.
\end{equation}

Considering $f$ as a polynomial in variables $z,\alpha_{1},\alpha_{2},\cdot \cdot \cdot,\alpha_{p^{e}-1},\beta$, the coeffiecients of $f$ are not necessarily integers. However the coefficients of $f$ belongs to $\Z_{(p)}$, the localisation of $\Z$ at the prime $p$. The polynomial $f$ is monic as a polynomial in  $\Z_{(p)}[z,\alpha_{1},\alpha_{2},\cdot \cdot \cdot,\alpha_{p^{e}-1}][\beta]$.
Similarly, the polynomials $f_{k,n,p^{e}}, h_{k,n,p^{e}}$ lie in the polynomial ring $\Z_{(p)}[p^{e}-1]:= \Z_{(p)}[\alpha_{1},\cdot \cdot \cdot,\alpha_{p^{e}-1},\beta]$,\;$\Z_{(p)}[p^{e}-2]$, respectively and both are monic as polynomials in $\beta$.

Note that,
\begin{align}\label{j}
    f(z)-f(w) & = (z^{p^{e}}-w^{p^{e}}) + \sum_{i=2}^{p^{e}-1} (-1)^{i} \frac{p^{e}}{p^{e}-i}\cdot s_{i}(\bar{\alpha})\cdot (z^{p^{e}-i}-w^{p^{e}-i}) \notag \\
    & =(z-w)\left(\Phi_{p^{e}}(z,w)+ \sum_{i=2}^{p^{e}-1} (-1)^{i}\frac{p^{e}}{p^{e}-i}\cdot s_{i}(\bar{\alpha})\cdot \Phi_{p^{e}-i}(z,w)\right),
\end{align}
where $\Phi_{d}(z,w)= (z^{d}-w^{d})/(z-w), d\in \N$.

\subsection{A weak version of Thurston's rigidity theorem for $\mathcal{M}_{p^{e}}$}\label{lemmasthurs}
Thurston's rigidity theorem studies transversality of the intersection of equations for two different critical points to be preperiodic, in moduli spaces of rational maps (see \cite{10.1007/BF02392534}). In this article, we need a much weaker version of Thurston's rigidity for moduli spaces of polynomials with a marked critical point. While the proof of the general Thurston's rigidity theorem requires deep tools, the version we need here (Theorem \ref{thurs}) can be proved using elementary algebraic methods. Let $f$ be the polynomial as defined in Equation \eqref{i}.
\begin{theorem}\label{thurs}
    Fix $i\in \{2,\cdot \cdot \cdot,p^{e}-1\}$. Fix $k_{1},k_{i}\in \N \cup \{0\}$, and $n_{1},n_{i}\in \N$. Then, the polynomials 
    \begin{equation*}
        f^{k_{1}+n_{1}}(\alpha_{1}) - f^{k_{1}}(\alpha_{1}) \text{ and } f^{k_{i}+n_{i}}(\alpha_{i}) - f^{k_{i}}(\alpha_{i})
    \end{equation*}
    are coprime in $\C[p^{e}-1]:= \C[\alpha_{1},\alpha_{2},\cdot \cdot \cdot,\alpha_{p^{e}-1},\beta]$. Moreover, the polynomials 
    \begin{equation*}
    \hat{f}^{k_{1}+n_{1}}(\hat{\alpha_{1}})-\hat{f}^{k_{1}}(\hat{\alpha_{1}}) \text{ and }\hat{f}^{k_{i}+n_{i}}(\hat{\alpha_{i}})- \hat{f}^{k_{i}}(\hat{\alpha_{i}})    
    \end{equation*}
    are coprime in $\C[p^{e}-2]$. 
\end{theorem}
\begin{proof}
    As all the polynomials in the statement of the theorem are defined over $\Q$, it is enough to show that they are coprime over $\Q$.

    More precisely, all the polynomials mentioned above are defined over $\Z_{(p)}$. Let $l\in \N$. Consider $f^{l}(\alpha_{j})$, where $j \in \{1,2,\cdot \cdot \cdot,p^{e}-1\}$. The term in $f^{l}(\alpha_{j})$ containing the highest power of $\beta$ is $\beta^{p^{(l-1)e}}$. As $n_{1},n_{i} \in \N$, we have $k_{1}+n_{1}>k_{1}\geq 0$ and $k_{i}+n_{i}>k_{i}\geq 0$. Hence, $f^{k_{1}+n_{1}}(\alpha_{1}) - f^{k_{1}}(\alpha_{1})$ and $f^{k_{i}+n_{i}}(\alpha_{i}) - f^{k_{i}}(\alpha_{i})$ are both monic polynomials in $\beta$ over the integral domain $\Z_{(p)}[\alpha_{1},\alpha_{2},\cdot \cdot \cdot,\alpha_{p^{e}-1}]$. Hence, to show that $f^{k_{1}+n_{1}}(\alpha_{1}) - f^{k_{1}}(\alpha_{1})$ and $f^{k_{i}+n_{i}}(\alpha_{i}) - f^{k_{i}}(\alpha_{i})$ are coprime over $\Q$, it is enough to show that they are coprime over $\Z_{(p)}$. 

    Observe that $\;\hat{}\;$ being a ring homomorphism, 
    \begin{align*}
        \hat{f}^{k_{1}+n_{1}}(\hat{\alpha_{1}})-\hat{f}^{k_{1}}(\hat{\alpha_{1}})& = \widehat{f^{k_{1}+n_{1}}(\alpha_{1}) - f^{k_{1}}(\alpha_{1})}\\ 
    \mbox{and} \quad  \hat{f}^{k_{i}+n_{i}}(\hat{\alpha_{i}})- \hat{f}^{k_{i}}(\hat{\alpha_{i}}) &= \widehat{f^{k_{i}+n_{i}}(\alpha_{i}) - f^{k_{i}}(\alpha_{i})}.
    \end{align*} As $\hat{\beta}=\beta$, we get that $\hat{f}^{k_{1}+n_{1}}(\hat{\alpha_{1}})-\hat{f}^{k_{1}}(\hat{\alpha_{1}})$ and $\hat{f}^{k_{i}+n_{i}}(\hat{\alpha_{i}})- \hat{f}^{k_{i}}(\hat{\alpha_{i}})$ are both monic polynomials in $\beta$ over $\Z_{(p)}[\alpha_{1},\alpha_{2},\cdot \cdot \cdot,\alpha_{p^{e}-2}]$. Similarly as above, to show that $\hat{f}^{k_{1}+n_{1}}(\hat{\alpha_{1}})-\hat{f}^{k_{1}}(\hat{\alpha_{1}})$ and $\hat{f}^{k_{i}+n_{i}}(\hat{\alpha_{i}})- \hat{f}^{k_{i}}(\hat{\alpha_{i}})$ are coprime over $Q$, it is enough to show that they are coprime over $\Z_{(p)}$.

    Now, we will show that $f^{k_{1}+n_{1}}(\alpha_{1}) - f^{k_{1}}(\alpha_{1})$ and $f^{k_{i}+n_{i}}(\alpha_{i}) - f^{k_{i}}(\alpha_{i})$ are coprime over $\Z_{(p)}$. As they are both monic polynomials in $\beta$ over $\Z_{(p)}[\alpha_{1},\alpha_{2},\cdot \cdot \cdot,\alpha_{p^{e}-1}]$, their g.c.d. must be monic polynomial in $\beta$ over $\Z_{(p)}[\alpha_{1},\alpha_{2},\cdot \cdot \cdot,\alpha_{p^{e}-1}]$ too, upto associates. Hence, $f^{k_{1}+n_{1}}(\alpha_{1}) - f^{k_{1}}(\alpha_{1})$ and $f^{k_{i}+n_{i}}(\alpha_{i}) - f^{k_{i}}(\alpha_{i})$ are coprime over $\Z_{(p)}$ if they are coprime modulo $p\Z_{(p)}$. Observe that,
    \begin{equation*}
        f(z)\equiv z^{p^{e}} - \alpha_{1}^{p^{e}} + \beta \;(\text{mod } p).
    \end{equation*}
    So,  $f^{l}(\alpha_{1})= \alpha_{1}\;(\text{or, }\beta)$ if $l=0\;(\text{or, }1)$, and for $l>1$,
    \begin{align*}
        f^{l}(\alpha_{1}) & \equiv (\beta^{p^{(l-1)e}} - \alpha_{1}^{p^{(l-1)e}}) + \cdot \cdot \cdot + (\beta^{p^{e}} - \alpha_{1}^{p^{e}}) + \beta\\
        & \equiv \beta + \sum_{j=1}^{l-1} (\beta- \alpha_{1})^{p^{je}} \;(\text{mod } p).
    \end{align*}
    Hence, 
    \begin{equation*}
        f^{k_{1}+n_{1}}(\alpha_{1}) - f^{k_{1}}(\alpha_{1}) \equiv \sum_{j=k_{1}}^{k_{1}+n_{1}-1} (\beta - \alpha_{1})^{p^{je}} \;(\text{mod }p).
    \end{equation*}
    Similarly, $f^{0}(\alpha_{i})= \alpha_{i}$, and for $l\in \N$
    \begin{equation*}
        f^{l}(\alpha_{i}) \equiv (\alpha_{i}-\alpha_{1})^{p^{le}} + (\beta - \alpha_{1})^{p^{(l-1)e}} + \cdot \cdot \cdot + (\beta - \alpha_{1})^{p^{e}} + \beta \;(\text{mod }p).
    \end{equation*}
    So modulo $p$, 
    \begin{equation*}
        f^{k_{i}+n_{i}}(\alpha_{i}) - f^{k_{i}}(\alpha_{i}) = \left( (\alpha_{i}-\alpha_{1})^{p^{(k_{i}+n_{i})e}} - (\alpha_{i} - \alpha_{1})^{p^{k_{i}e}}\right) + \left( \sum_{j=k_{i}}^{k_{i}+n_{i}-1} (\beta - \alpha_{1})^{p^{je}} \right).
    \end{equation*}
    Modulo $p$, the highest degree homogeneous part of $f^{k_{1}+n_{1}}(\alpha_{1}) - f^{k_{1}}(\alpha_{1})$ is 
    $$(\beta - \alpha_{1})^{p^{(k_{1}+n_{1}-1)e}}$$
    and the highest degree homogeneous part of $f^{k_{i}+n_{i}}(\alpha_{i}) - f^{k_{i}}(\alpha_{i})$ is $$(\alpha_{i}-\alpha_{1})^{p^{(k_{i}+n_{i})e}}.$$ The highest degree homogeneous part of g.c.d. of two polynomials must divide the highest degree homogeneous part of each of them. As $(\beta- \alpha_{1})$ and $(\alpha_{i}-\alpha_{1})$ are two distinct irreducible polynomials over $\F_{p}$, the polynomials $f^{k_{1}+n_{1}}(\alpha_{1}) - f^{k_{1}}(\alpha_{1})$ and $f^{k_{i}+n_{i}}(\alpha_{i}) - f^{k_{i}}(\alpha_{i})$ are coprime modulo $p\Z_{(p)}$. Hence, They are coprime over $\Z_{(p)}$ too. 

    In the last paragraph, replacing $f^{k_{1}+n_{1}}(\alpha_{1}) - f^{k_{1}}(\alpha_{1})$ and $f^{k_{i}+n_{i}}(\alpha_{i}) - f^{k_{i}}(\alpha_{i})$ with $\hat{f}^{k_{1}+n_{1}}(\hat{\alpha_{1}})-\hat{f}^{k_{1}}(\hat{\alpha_{1}})$ and $\hat{f}^{k_{i}+n_{i}}(\hat{\alpha_{i}})- \hat{f}^{k_{i}}(\hat{\alpha_{i}})$, respectively, every argument follows verbatim and one obtains the second part of the theorem.
\end{proof}
\begin{remark}
    We will require the above theorem in later sections. Whenever we mention ``Thurston's rigidity" from here on, we will mean the above theorem (Theorem \ref{thurs}).
\end{remark}
\subsection{Generalised Eisenstein Irreducibility criterion}
\begin{theorem}\label{m}
 Let $p,e\in \N, p$ odd prime. Let $\Z_{(p)}$ be the localisation of $\Z$ at the prime $p$. Let $g,h$ be non-constant elements of $\Z_{(p)}[p^{e}-2]$, both monic as elements of $\Z_{(p)}[\alpha_{1},\alpha_{2},\cdot \cdot \cdot,\alpha_{p^{e}-2}][\beta]$. Let $\mbox{Res}(g,h)$ denote the resultant of $g,h$ both considered as polynomials in $\beta$ over the integral domain $\Z_{(p)}[\alpha_{1},\alpha_{2},\cdot \cdot \cdot,\alpha_{p^{e}-2}]$. Suppose the following conditions hold,\\
 1) $g\equiv h^{n}\;(\text{mod }p)$, for some $n\in \N$.\\
2) $h\;(\text{mod } p)$ is irreducible in $\F_{p}[p^{e}-2]$.\\
3) $\mbox{Res}(g,h)\not\equiv 0\;(\text{mod }p^{2\cdot deg(h)}\Z_{(p)})$, where $deg(h)$ is the degree of $h$ as a polynomial in $\beta$ over $\Z_{(p)}[\alpha_{1},\alpha_{2},\cdot \cdot \cdot,\alpha_{p^{e}-2}]$.

Then, $g$ is irreducible in $\Q[p^{e}-2]$.
\end{theorem}
\begin{proof}
We will prove by contradiction. Let us assume that $g$ is reducible in\\ $\Q[p^{e}-2]$. As $g$ is monic in $\beta$, by Gauss lemma, let $g=g_{1}\cdot g_{2},$ such that $ g_{1},g_{2}\in \Z_{(p)}[p^{e}-2]$ both monic as elements of $\Z_{(p)}[\alpha_{1},\alpha_{2},\cdot \cdot \cdot,\alpha_{p^{e}-2}][\beta]$ and none of them is a unit. By conditions $1$ and $2$ of the theorem, $g_{1}\equiv h^{n_{1}}\;\modp$, and $g_{2}\equiv h^{n_{2}}\;\modp$, such that $n_{1},n_{2}\in \N, n_{1}+n_{2}=n$. So, there exist polynomials $x_{1},x_{2}\in \Z_{(p)}[p^{e}-2]$, such that $g_{1}=h^{n_{1}}+p\cdot x_{1}$ and $g_{2}=h^{n_{2}}+p\cdot x_{2}$. Then,
\begin{align*}
    \mbox{Res}(g,h) & =\mbox{Res}(g_{1},h)\cdot \mbox{Res}(g_{2},h)\\
    & =\mbox{Res}(p\cdot x_{1},h)\cdot \mbox{Res}(p\cdot x_{2},h)\\
     & = p^{2\cdot deg(h)}\cdot \mbox{Res}(x_{1},h)\cdot \mbox{Res}(x_{2},h).
\end{align*}
 But by the third assumption of the theorem, $p^{2\cdot deg(h)}$ does not divide the $\mbox{Res}(g,h)$. Hence, we arrive at a contradiction. 
\end{proof}

\section{Irreducibility of $\Sigma_{k,1,p}$}\label{t'''}
Let the degree of the polynomial $f$ in Equation \eqref{b} be an odd prime $p$. Then,
\begin{equation}\label{fdegp}
    f(z)=z^{p} + \sum_{i=2}^{p-1}(-1)^{i}\frac{p\cdot s_{i}(\bar{\alpha})}{p-i} z^{p-i} - \alpha_{1}^{p} - \sum_{i=2}^{p-1}(-1)^{i}\frac{p\cdot s_{i}(\bar{\alpha})}{p-i} \alpha_{1}^{p-i} + \beta.
\end{equation}
 We recall a few notations and a fact from previous sections.  
\begin{itemize}
    \item For any ring $R$ and $d\in \N$, we define $R[d]:= R[\alpha_{1}, \alpha_{2},.., \alpha_{d}, \beta]$.
    \item $f_{k,1,p}$ lies in $ \Z_{(p)}[p-1],\hspace{5pt} \hat{f}_{k,1,p}, h_{k,1,p}$ lie in $\Z_{(p)}[p-2]$.
    \item for $g,h\in \Z_{(p)}[p-1]$,\hspace{5pt} $g\equiv h\;\modp$ implies that $\hat{g} \equiv \hat{h}\;\modp$.
\end{itemize}

Let $k\in \N$. We will show that $h_{k,1,p}$ is a generalised Eisenstein polynomial. First, we study the polynomial $h_{0,1,p}$.
\begin{lemma}\label{h01p}
The polynomial $h_{0,1,p}$ is $\beta - \alpha_{1}$.
\end{lemma}
\begin{proof}
By definition of $h_{0,1,p}$, we have $h_{0,1,p}=\hat{f}_{0,1,p}= \hat{f}(\alpha_{1})-\alpha_{1}=\beta - \alpha_{1}$.
\end{proof}
Now we will show $h_{k,1,p}$ satisfies condition $1$ of generalised Eisenstein irreducibility criterion (Theorem \ref{m}), with respect to the polynomial $h_{0,1,p}$.
\begin{lemma}\label{Eisen1}
For any $k\in \N$, $h_{k,1,p}\equiv h_{0,1,p}^{N_{k,p}}\:\modp$, for some $N_{k,p}\in \N$.
\end{lemma}
\begin{proof}
Let $k\in \N$. From Equation \eqref{fdegp}, we have $f(z)\equiv z^{p}-\alpha_{1}^{p}+\beta$ $\modp$. So, 
\begin{equation*}
    f^{k}(z)\equiv z^{p^{k}}-\alpha_{1}^{p^{k}}+\beta^{p^{k-1}}-\alpha_{1}^{p^{k-1}}+\cdot \cdot \cdot+\beta^{p}-\alpha_{1}^{p}+\beta\;\; \modp,
\end{equation*}
\begin{align*}
     f_{k,1,p} & = f^{k+1}(\alpha_{1})-f^{k}(\alpha_{1})\\
     & = f^{k}(\beta) -f^{k-1}(\beta)\\ & \equiv \beta^{p^{k}}-\alpha_{1}^{p^{k}} \modp\\
     & \equiv (\beta-\alpha_{1})^{p^{k}} \modp.
\end{align*}
Hence,  $\hat{f}_{k,1,p} = \widehat{f_{k,1,p}} \equiv (\beta - \alpha_{1})^{p^{k}} \;\modp $. As $h_{k,1,p}$ divides $\hat{f}_{k,1,p}$ and $h_{0,1,p}= \beta-\alpha_{1}$ is irreducible over $\F_{p}$, we get $h_{k,1,p}\equiv h_{0,1,p}^{N_{k,p}}\; \modp$, for some $N_{k,p}\in \N$. This proves the lemma.   
\end{proof}
We see from Lemma \ref{h01p} that $h_{0,1,p}$ satisfies condition $2$ of generalised Eisenstein irreducibility criterion (Theorem \ref{m}). For the resultant criterion (condition $3$) of Theorem \ref{m}, we need to study the lowest degree homogeneous part of $h_{k,1,p}$. 
\begin{lemma}\label{Eisen2}
For any natural number $k\geq 2$, the polynomial $h_{k,1,p}$ is of the form  
\begin{equation*}
    h_{k,1,p}= \frac{\hat{f}_{k,1,p}}{\hat{f}_{k-1,1,p}\cdot (\beta-\alpha_{1})},
\end{equation*}
 and the lowest degree homogeneous part of $h_{k,1,p}$ is $p(\beta + \sum_{i=1}^{p-2} \alpha_{i})\prod_{i=2}^{p-2} (\beta-\alpha_{i})$.
\end{lemma}
\begin{proof}
To obtain $h_{k,1,p}$, we need to factor out all the irreducible factors of $\hat{f}_{k-1,1,p}$ from $\hat{f}_{k,1,p}$, each raised to the highest power that divides $\hat{f}_{k,1,p}$. First, let us consider the polynomial $$\frac{\hat{f}_{k,1,p}}{\hat{f}_{k-1,1,p}} = \frac{\widehat{f_{k,1,p}}}{\widehat{f_{k-1,1,p}}}.$$

From Equation \eqref{j},
\begin{equation}\label{frac}
    \frac{f(z)-f(w)}{z-w}=\Phi_{p}(z,w)+ \sum_{i=2}^{p-1} (-1)^{i}\frac{p}{p-i}\cdot s_{i}(\bar{\alpha})\cdot \Phi_{p-i}(z,w).
\end{equation}
Putting $z=f^{k}(\alpha_{1}), w=f^{k-1}(\alpha_{1})$ in above, we get,
\begin{equation*}
    \frac{\hat{f}_{k,1,p}}{\hat{f}_{k-1,1,p}}=\Phi_{p}\left(\hat{f}^{k}(\alpha_{1}),\hat{f}^{k-1}(\alpha_{1})\right)+ \sum_{i=2}^{p-1} (-1)^{i}\frac{p}{p-i}\cdot \widehat{s_{i}(\bar{\alpha})}\cdot \Phi_{p-i}\left(\hat{f}^{k}(\alpha_{1}),\hat{f}^{k-1}(\alpha_{1})\right).
\end{equation*}

Next, we need to figure out what are the irreducible polynomials in $\Z_{(p)}[p-2]$ that divide both $\hat{f}_{k-1,1,p}$ and $\hat{f}_{k,1,p}/\hat{f}_{k-1,1,p}$. Note that,
\begin{equation}\label{der}
    \frac{f(z)-f(w)}{z-w} \equiv f'(w) \;(\text{mod } z-w),
\end{equation}
where $f'(w)$ is the derivative of $f(w)$ with respect to $w$. Replacing $z$ with $f^{k}(\alpha_{1})$ and $w$ with $f^{k-1}(\alpha_{1})$ in Equation \eqref{der}, we get
\begin{equation*}
    \frac{f_{k,1,p}}{f_{k-1,1,p}}\equiv f'\left(f^{k-1}(\alpha_{1})\right) \equiv p\left(\prod_{i=1}^{p-1}\left(f^{k-1}(\alpha_{1})-\alpha_{i}\right)\right) \text{ (mod }f_{k-1,1,p}). 
\end{equation*}
\begin{equation}
    \label{eq:5.4}
    \text{Hence, } \frac{\hat{f}_{k,1,p}}{\hat{f}_{k-1,1,p}} \equiv p\left(\prod_{i=1}^{p-1}\left(\hat{f}^{k-1}(\alpha_{1})-\hat{\alpha_{i}}\right)\right) \text{ (mod }\hat{f}_{k-1,1,p}).
\end{equation}

Let $g$ be an irreducible polynomial in $\Z_{(p)}[p-2]$ that divides both $\hat{f}_{k-1,1,p}$ and $\hat{f}_{k,1,p}/\hat{f}_{k-1,1,p}$. Then, $g$ divides $ p\prod_{i=1}^{p-1}(\hat{f}^{k-1}(\alpha_{1})-\hat{\alpha_{i}})$. By Thurston's rigidity theorem (Theorem \ref{thurs}), $\hat{f}_{k-1,1,p}$ and $\hat{f}^{k-1}(\alpha_{1})-\hat{\alpha_{i}}$ are coprime over $\Z_{(p)}$, for any $i, 2\leq i\leq p-1$. So, $g$ divides $\hat{f}^{k-1}(\alpha_{1})-\hat{\alpha_{1}}$ (we can remove $p$, as all the polynomials considered here are monic in $\beta$ over $\Z_{(p)}[\alpha_{1},\alpha_{2},\cdot \cdot \cdot,\alpha_{p-2}]$). Note that, $\hat{f}^{k-1}(\alpha_{1})-\hat{\alpha_{1}}=\hat{f}_{0,k-1,p}$. As $g$ divides both $\hat{f}_{k-1,1,p}$ and $\hat{f}_{0,k-1,p}$, from Lemma \ref{k}, $g$ divides $\hat{f}_{0,1,p}=\beta-\alpha_{1}$. Assuming $t_{k}\in \Z_{\geq 0}$ is the highest power of $\beta-\alpha_{1}$ that divides $\hat{f}_{k,1,p}/\hat{f}_{k-1,1,p}$, we get
\begin{equation*}
    h_{k,1,p}=\frac{\hat{f}_{k,1,p}}{\hat{f}_{k-1,1,p}(\beta-\alpha_{1})^{t_{k}}}, \:\:\: \text{ for any }k\in \N.
\end{equation*}

We will now show that $\beta-\alpha_{1}$ divides $f_{k,1,p}/f_{k-1,1,p}$. Note that,
\begin{equation*}
    f^{k}(\alpha_{1}) \equiv \alpha_{1} \;\;(\text{mod } \beta-\alpha_{1}), \text{ for any } k\in \N.
\end{equation*}
Therefore, modulo $\beta-\alpha_{1}$,
\begin{align*}
    \frac{f_{k,1,p}}{f_{k-1,1,p}} & =\Phi_{p}\left(f^{k}(\alpha_{1}),f^{k-1}(\alpha_{1})\right)+ \sum_{i=2}^{p-1} (-1)^{i}\frac{p}{p-i} s_{i}(\bar{\alpha}) \Phi_{p-i}\left(f^{k}(\alpha_{1}),f^{k-1}(\alpha_{1})\right)\\
    & = \Phi_{p}(\alpha_{1},\alpha_{1})+ \sum_{i=2}^{p-1} (-1)^{i}\frac{p}{p-i} s_{i}(\bar{\alpha}) \Phi_{p-i}(\alpha_{1},\alpha_{1})\\
    & = p\left(\alpha_{1}^{p-1}+\sum_{i=2}^{p-1} (-1)^{i} s_{i}(\bar{\alpha})\alpha_{1}^{p-i-1}\right)\\
    &= f'(\alpha_{1})\\
    & = 0 
\end{align*}
Hence, $\beta-\alpha_{1}$ divides $f_{k,1,p}/f_{k-1,1,p}$ in $\Z_{(p)}[p-1]$.

Let us compute the lowest degree homogeneous part of $f_{k,1,p}/\left(f_{k-1,1,p}(\beta-\alpha_{1})\right)$.
Observe that, the lowest degree homogeneous part of $f^{l}(\alpha_{1})$ is $\beta$, for any $l\in \N$. So, for any $k\in \N, k\geq 2$ (note that in this proof this is the first time we excluded the case $k=1$), the lowest degree homogeneous part of $f_{k,1,p}/f_{k-1,1,p}$ is
\begin{equation*}
    \Phi_{p}(\beta,\beta)+ \sum_{i=2}^{p-1} (-1)^{i}\frac{p}{p-i}\cdot s_{i}(\bar{\alpha})\cdot \Phi_{p-i}(\beta,\beta) = f'(\beta)= p\prod_{i=1}^{p-1}(\beta-\alpha_{i}).
\end{equation*}
Hence, the lowest degree homogeneous part of $f_{k,1,p}/\left(f_{k-1,1,p}(\beta-\alpha_{1})\right)$ is $p\prod_{i=2}^{p-1}(\beta-\alpha_{i})$ and the lowest degree homogeneous part of $\hat{f}_{k,1,p}/\left(\hat{f}_{k-1,1,p}(\beta-\alpha_{1})\right)$ is 
\begin{equation*}
    p\prod_{i=2}^{p-1}(\beta-\hat{\alpha_{i}}) = p(\beta + \sum_{i=1}^{p-2} \alpha_{i})\prod_{i=2}^{p-2} (\beta-\alpha_{i}).
\end{equation*}
 As $\beta-\alpha_{1}$ does not divide the lowest degree homogeneous part of $\hat{f}_{k,1,p}/\left(\hat{f}_{k-1,1,p}(\beta-\alpha_{1})\right)$, we get that $\beta-\alpha_{1}$ does not divide $\hat{f}_{k,1,p}/\left(\hat{f}_{k-1,1,p}(\beta-\alpha_{1})\right)$. Hence, both parts of the lemma are proved.
\end{proof}
As a corollary of Lemma \ref{Eisen2}, we get that the resultant condition of generalised Eisenstein irreducibility criterion (Theorem \ref{m}) is satisfied by $h_{k,1,p}$ and $h_{0,1,p}$, for any $k\geq 2$.
\begin{corollary}\label{resk1p}
For any $k\in \N, k\geq 2$, the resultant $\mbox{Res}\;(h_{k,1,p},h_{0,1,p})\not\equiv 0 \; (\text{mod }p^{2})$.
\end{corollary}
\begin{proof}
Let $k\in \N, k\geq 2$. The resultant of $h_{k,1,p}$ and $h_{0,1,p}$ is
$$\mbox{Res}(h_{k,1,p},h_{0,1,p})=\mbox{Res}(h_{k,1,p},\beta-\alpha_{1})= (-1)^{\epsilon} h_{k,1,p}(\alpha_{1},\alpha_{2},\cdot \cdot \cdot,\alpha_{p-2},\alpha_{1}),$$ 
where $\epsilon$ is the degree of $h_{k,1,p}$ as a polynomial in $\beta$.
From Lemma \ref{Eisen2}, the lowest degree homogeneous part of $\mbox{Res}(h_{k,1,p},h_{0,1,p})$ is 
$$(-1)^{\epsilon} p(\alpha_{1} + \sum_{i=1}^{p-2} \alpha_{i})\prod_{i=2}^{p-2} (\alpha_{1}-\alpha_{i})\not\equiv 0 \;\;\;(\text{mod } p^{2}).$$ 
Hence, the corollary is proved.
\end{proof}
Next we will study the resultant $\mbox{Res}(h_{1,1,p},h_{0,1,p})$. For that, we first need to figure out what $h_{1,1,p}$ looks like.
\begin{lemma}\label{h11p}
    The following equality holds,
    \begin{equation*}
        h_{1,1,p}= \frac{\hat{f}_{1,1,p}}{(\beta-\alpha_{1})^{2}}.
    \end{equation*}
\end{lemma}
\begin{proof}
From the proof of Lemma \ref{Eisen2}, we know that $h_{1,1,p}$ divides 
$$\frac{\hat{f}_{1,1,p}}{\hat{f}_{0,1,p}(\beta-\alpha_{1})} = \frac{\hat{f}_{1,1,p}}{(\beta-\alpha_{1})^{2}}.$$

We need to check whether higher powers of $(\beta - \alpha_{1})$ divides $\hat{f}_{1,1,p}$.  As $f'(\alpha_{1})= 0$, we get from Taylor series expansion $f$ at $\alpha_{1}$,
$$\frac{f(z)-f(\alpha_{1})}{(z-\alpha_{1})^{2}} \equiv \frac{f''(\alpha_{1})}{2} \;(\text{mod } z-\alpha_{1}).$$

Putting $z=\beta$ in above, 
\begin{equation}\label{5.5}
    \frac{f_{1,1,p}}{(\beta-\alpha_{1})^{2}} = \frac{f(\beta)-f(\alpha_{1})}{(\beta - \alpha_{1})^{2}} \equiv \frac{f''(\alpha_{1})}{2} \;(\text{mod } \beta - \alpha_{1}).
\end{equation}

As $f''(\alpha_{1})$ is not identically zero, we get from Equation \eqref{5.5},
\begin{equation}\label{5.6}
    h_{1,1,p} = \frac{\hat{f}_{1,1,p}}{(\beta-\alpha_{1})^{2}}.
\end{equation}
\end{proof}

\begin{remark}
    Brief calculation shows that explicit form of $h_{1,1,p}$ is as follows,
    $$h_{1,1,p}:=\sum_{j=0}^{p-1}\alpha_{1}^{j}\Phi_{p-1-j}(\beta,\alpha_{1})+ \sum_{i=2}^{p-1}(-1)^{i}\frac{p\cdot s_{i}(\bar{\alpha})}{p-i}\sum_{j=0}^{p-i-1}\alpha_{1}^{j}\Phi_{p-i-1-j}(\beta,\alpha_{1}).$$ We will not require this explicit form in what follows.
\end{remark}

Now, we can study the resultant of $h_{1,1,p}$ and $h_{0,1,p}$.
\begin{corollary}\label{res11p}
The resultant $\mbox{Res}(h_{1,1,p},h_{0,1,p})\not\equiv 0\;(\text{mod }p^{2})$.
\end{corollary}
\begin{proof}
    From Equations \eqref{5.5} and \eqref{5.6}, $\mbox{Res}(h_{1,1,p},h_{0,1,p})= -\frac{\widehat{f''(\alpha_{1})}}{2}\not\equiv 0 \;(\text{mod }p^{2})$.
\end{proof}
We are ready to prove irreducibility of $h_{k,1,p}$ over $\Q$.
\begin{theorem}\label{irrq}
For any $k\in \N$, the polynomial $h_{k,1,p}$ is irreducible over $\Q$.
\end{theorem}
\begin{proof}
 Fix $e=1,\; g=h_{k,1,p},\; h=h_{0,1,p}=\beta-\alpha_{1}$, in the generalised Eisenstein irreducibility criterion (Theorem \ref{m}). The Lemmas \ref{h01p}, \ref{Eisen1} and Corollaries \ref{resk1p}, \ref{res11p} show that all the three conditions of Theorem \ref{m} holds. Hence, $h_{k,1,p}$ is $p$-Eisenstein with respect to the polynomial $h_{0,1,p}$, for every $k\in \N$. So, $h_{k,1,p}$ is irreducible over $\Q$, for every choice of $k\in \N$. 
\end{proof}
Next, we extend the irreducibility of $h_{k,1,p}$ from over $\Q$ to over $\C$, which, by Corollary \ref{coro1}, is enough to prove the following theorem.
\begin{theorem}\label{irrc}
For any $k\in \N\cup \{0\}$, the set $\Sigma_{k,1,p}$ is an irreducible quasi-affine variety.
\end{theorem}
\begin{proof}
 By Corollary \ref{coro1}, it is enough to show that $h_{k,1,p}$ is irreducible over $\C$.
 
For $k=0$, this is trivial.

Let $k=1$. Recall that the derivative of $f$ is of the form, 
$$f'(z)= p \prod_{i=1}^{p-1} (z - \alpha_{i}).$$
From Equation \eqref{5.6}, we get 
\begin{equation*}
    h_{1,1,p}\equiv \frac{\widehat{f''(\alpha_{1})}}{2}\equiv \frac{p}{2} \prod_{i=2}^{p-1}(\alpha_{1}-\hat{\alpha_{i}})\;(\text{mod }\beta-\alpha_{1}).
\end{equation*}
So, 
 \begin{equation*}
     h_{1,1,p}=\frac{p}{2}\prod_{i=2}^{p-1}(\beta-\hat{\alpha_{i}}) + (\beta-\alpha_{1})t_{p},
 \end{equation*}
for some polynomial $t_{p}\in \Z_{(p)}[p-2]$. Let us assume that $h_{1,1,p}$ is reducible over $\C$. Up to a rearrangement of $\alpha_{i}$'s, $2\leq i\leq p-1$, there exists an irreducible factor $g$ of $h_{1,1,p}$ of the form 
  \begin{equation*}
      g=\left(\prod_{i=2}^{l}(\beta-\hat{\alpha_{i}})\right)+(\beta-\alpha_{1})t,
  \end{equation*}
   for some $l\in \N,2\leq l\leq p-1, $ and $ t\in \C[p-2]$.

Consider $G_{\Q}$, the absolute Galois group of $\Q$. For any $\sigma \in G_{\Q}$, we have $\sigma(g)$ divides $h_{1,1,p}$. Observe that the first summand in the expression of $g$ above is defined over $\Q$. Hence, if $\sigma(g)$ is a constant multiple of $g$ then $\sigma(g)=g$. Otherwise, if $\sigma(g)\neq g$, then $g$ being irreducible, $g\cdot \sigma(g)$ divides $h_{1,1,p}$. But that is not possible as the first summand in the expression of $h_{1,1,p}$ has all simple factors and they are all defined over $\Q$. So, $\sigma(g)=g$. As $\sigma \in G_{\Q}$ was arbitrarily chosen, we get that $g\in \Q[p-2]$. But, $h_{1,1,p}$ is irreducible over $\Q$, by Theorem \ref{irrq}. We arrive at a contradiction. Hence, $h_{1,1,p}$ is irreducible over $\C$.

For $k \geq 2$, from Lemma \ref{Eisen2}, we have that the lowest degree homogeneous part of $h_{k,1,p}$ is 
\begin{equation*}
    p\prod_{i=2}^{p-1}(\beta-\hat{\alpha_{j}}).
\end{equation*}
 Observe that the lowest degree homogeneous part of $h_{k,1,p}$ has all simple factors and they are all defined over $\Q$. Similarly as the $k=1$ case, one sees that any irreducible factor of $h_{k,1,p}$ must lie in $\Q[p-2]$. As $h_{k,1,p}$ is irreducible over $\Q$ by Theorem \ref{irrq}, it is also irreducible over $\C$.
\end{proof}

\section{mixed critical cases: The cases in between the unicritical and separably critical case}\label{mixcritical}
Let $k\geq 0, p>3, p$ prime. For a degree $p$ polynomial $f$ with marked critical point $\alpha_{1}$, 
we define \emph{ramification index} of the pair $(f, \alpha_{1})$ to be the ramification index of $\alpha_{1}$ under the map $f$, which is the multiplicity of $\alpha_{1}$ as a root of the equation $f(z)-f(\alpha_{1})$. As $\alpha_{1}$ is a critical point of a degree $p$ polynomial $f$, ramification index of the pair $(f,\alpha_{1})$ varies between $2$ and $p$. By definition, $\mathcal{M}_{p}$ is the space of affine conjugacy classes of pairs $(f,\alpha_{1})$, where $f$ is a degree $p$ polynomial and $\alpha_{1}$ is one of its finite critical points. Ramification index stays invariant under affine conjugation of polynomials with a marked critical point. Hence, for any point in $\mathcal{M}_{p}$, its \emph{ramification index} is well-defined and it lies between $2$ and $p$.

For any point $(\alpha_{1},\alpha_{2},\cdot \cdot \cdot,\alpha_{p-2},\beta)\in \mathcal{M}_{p}$, we call it \emph{separably critical} if its finite critical points $\alpha_{1},\alpha_{2},\cdot \cdot \cdot,\alpha_{p-2}, \alpha_{p-1}:=-(\sum_{i=1}^{p-2}\alpha_{i})$ are all distinct and \emph{unicritical} if $\alpha_{1}=\alpha_{2}=\cdot \cdot \cdot.=\alpha_{p-1}=0$. A polynomial is separably critical if and only if any finite critical point of it has ramification index $2$. A polynomial in $\mathcal{M}_{p}$ is unicritical iff it has a unique finite critical point of ramification index $p$.

Observe that the set of all separably critical points of $\Sigma_{k,1,p}$ is non-empty and open in $\Sigma_{k,1,p}$. As $\Sigma_{k,1,p}$ is irreducible and quasi-affine, the set of all separably critical points in $\Sigma_{k,1,p}$ is also an irreducible quasi-affine variety. In this section, we will study similar irreducibility questions for mixed critical cases. A point in $\mathcal{M}_{p}$ is called \emph{mixed critical} iff its ramification index $e$ satisfies $2 < e < p$.

Fix $j,\; 2\leq j \leq p-2$. Let $\Sigma_{j}:= \Sigma_{k,1,p} \cap V_{j}, \text{ where }V_{j} := \{(\alpha_{1},\alpha_{2},\cdot \cdot \cdot,\alpha_{p-2},\beta) \in \mathcal{M}_{p}\;|\; \alpha_{1}=\alpha_{2}=\cdot \cdot \cdot=\alpha_{j}\}$. The set $\Sigma_{j}$ is the set of all points of $\Sigma_{k,1,p}$ for which the ramification index is $\geq j+1$. This is a quasi-affine subset of the affine variety $V_{j}$.  Here, we will show that $\Sigma_{j}$ is irreducible, for every $j,\; 2\leq j\leq p-2$.

\begin{lemma}
    \label{lem:f_k,1,p}
    For any $k \geq 1, p > 3, p$ prime and $2 \le j \le p-2$, the polynomial $\hat{f}_{k,1,p}$ is contained in the ideal $((\beta - \alpha_{1})^{j+1}, \alpha_{1} - \alpha_{2}, \alpha_{1} - \alpha_{3}, \ldots, \alpha_{1} - \alpha_{j})$.
\end{lemma}

\begin{proof}
    Observe that to prove the lemma, it is enough to show that $f_{k,1,p}$ is contained in the same ideal. As $f_{1,1,p}$ divides $f_{k,1,p}$ for all $k \ge 1$, it is enough to show that $f_{1,1,p}$ lies in the ideal $((\beta - \alpha_{1})^{j+1}, \alpha_{1} - \alpha_{2}, \alpha_{1} - \alpha_{3}, \ldots, \alpha_{1} - \alpha_{j})$. In other words, we want to show that modulo the ideal $(\alpha_{1} - \alpha_{2}, \alpha_{1} - \alpha_{3}, \ldots, \alpha_{1} - \alpha_{j})$, $f_{1,1,p}$ is divisible by $(\beta - \alpha_{1})^{j+1}$.

    Consider the Taylor series expansion of $f$ at $\alpha_{1}$ (note that $f'(\alpha_{1})=0$),

    \[
    f(z) = f(\alpha_{1}) + \frac{f''(\alpha_{1})}{2} (z-\alpha_{1})^{2} + \frac{f'''(\alpha_{1})}{3!} (z-\alpha_{1})^{3} + \ldots.
    \]

    As $\alpha_{1}, \alpha_{2}, \ldots, \alpha_{p-1}$ are critical points of $f(z)$, considering $f(z)$ modulo the ideal $(\alpha_{1}- \alpha_{2}, \alpha_{1} - \alpha_{3}, \ldots, \alpha_{1} - \alpha_{j})$ is the same as considering the ramification index of $\alpha_{1}$ in $f$ to be $j+1$. Thus, 
    \[
    f(z) \equiv f(\alpha_{1}) + \frac{f^{'(j+1)}(\alpha_{1})}{(j+1)!} (z-\alpha_{1})^{j+1} + \ldots \left(\text{mod } (\alpha_{1} - \alpha_{2}, \alpha_{1} - \alpha_{3}, \ldots, \alpha_{1} - \alpha_{j})\right).
    \]

    Replacing $z$ with $\beta$ in the equation above, we see

    \begin{equation}
        \label{eq:6.1}
    f_{1,1,p} = f(\beta) - f(\alpha_{1}) \equiv \frac{f^{'(j+1)}(\alpha_{1})}{(j+1)!} (\beta -\alpha_{1})^{j+1} \left(\text{mod } (\alpha_{1} - \alpha_{2}, \alpha_{1} - \alpha_{3}, \ldots, \alpha_{1} - \alpha_{j})\right).
    \end{equation}

    Hence, the lemma is proved.
\end{proof}

\begin{corollary}
    \label{cor:h_k,1,p}
    For any $k \in \N, p > 3, p$ prime and $2 \le j \le p-2$, the polynomial $h_{k,1,p}$ is contained in the ideal $((\beta - \alpha_{1})^{j-1}, \alpha_{1} - \alpha_{2}, \alpha_{1} - \alpha_{3}, \ldots, \alpha_{1} - \alpha_{j})$.
\end{corollary}

\begin{proof}
    For $k=1$, the lemma follows from Equation \eqref{eq:6.1} and Lemma \ref{h11p}. Let $k \geq 2$.
    From Equation \eqref{eq:5.4},
    \[
    \frac{\hat{f}_{k,1,p}}{\hat{f}_{k-1,1,p}} = p \left( \prod_{i=1}^{p-1} \left( \hat{f}^{k-1}(\alpha_{1}) - \hat{\alpha_{i}} \right) \right) + t \hat{f}_{k-1,1,p}, 
    \]
    for some polynomial $t$. From Lemma \ref{lem:f_k,1,p}, $\hat{f}_{k-1,1,p}$ is divisible by $(\beta- \alpha_{1})^{j+1}$, modulo the ideal $(\alpha_{1} - \alpha_{2}, \alpha_{1} - \alpha_{3}, \ldots, \alpha_{1} - \alpha_{j})$. On the other hand, modulo $(\alpha_{1} - \alpha_{2}, \alpha_{1} - \alpha_{3}, \ldots, \alpha_{1} - \alpha_{j})$,

    \[
    \prod_{i=1}^{p-1} \left( \hat{f}^{k-1}(\alpha_{1}) - \hat{\alpha_{i}} \right) \equiv \left( \hat{f}^{k-1}(\alpha_{1}) - \hat{\alpha_{1}} \right)^{j} \prod_{i=j+1}^{p-1} \left( \hat{f}^{k-1}(\alpha_{1}) - \hat{\alpha_{i}} \right)
    \]

    As $\hat{f}_{0,1,p}=(\beta - \alpha_{1})$ divides $\hat{f}_{0,k-1,p} = \hat{f}^{k-1}(\alpha_{1}) - \hat{\alpha_{1}}$, we see

    \[
    h_{k,1,p} = \frac{\hat{f}_{k,1,p}}{\hat{f}_{k-1,1,p} (\beta - \alpha_{1})} \in \left( (\beta-\alpha_{1})^{j-1}, \alpha_{1} - \alpha_{2}, \alpha_{1} - \alpha_{3}, \ldots, \alpha_{1} - \alpha_{j} \right). 
    \]
\end{proof}

\begin{theorem}\label{inbetween}
   Let $k \in \N, p \text{ prime }, p>3, 2\leq j\leq p-2$. The set $\Sigma_{j}:= \Sigma_{k,1,p} \cap V_{j}$ is an irreducible quasi-affine subvariety of $\mathcal{M}_{p}$.
\end{theorem}
\begin{proof}
Similar to the separably critical case, to prove irreducibility of $\Sigma_{j}$ it is enough to show that the ideal $$(h_{k,1,p}, \alpha_{1}-\alpha_{2},\alpha_{1}-\alpha_{3},\cdot \cdot \cdot,\alpha_{1}-\alpha_{j}),$$
is prime in $\C[\alpha_{1},\alpha_{2},\cdot \cdot \cdot,\alpha_{p-2},\beta]$.
For that, it is enough to show that the polynomial $$h_{k,1,p}^{j}:= h_{k,1,p}(\alpha_{1},\alpha_{1},\cdot \cdot \cdot,\alpha_{1},\alpha_{j+1},\cdot \cdot \cdot,\alpha_{p-2},\beta),$$ 
is irreducible in $\C[\alpha_{1},\alpha_{j+1},\alpha_{j+2},\cdot \cdot \cdot,\alpha_{p-2}, \beta]$.

Write $h^{j}_{k,1,p}$ as a polynomial in $(\beta - \alpha_{1})$ with coefficients from $\Q[\alpha_{1}, \alpha_{j+1}, \ldots, \alpha_{p-2}, \beta]$. Now, $h_{k,1,p}$ being $p$-Eisenstein with respect to $(\beta - \alpha_{1})$, all non-leading coefficients of the monic polynomial $h^{j}_{k,1,p}$ as a polynomial in $(\beta - \alpha_{1})$ is divisible by $p$. From Lemma~\ref{Eisen2}, the lowest degree part of $h^{j}_{k,1,p}$ is 
\[
p \left( \beta - \alpha_{1} \right)^{j-1} \left( \prod_{i=j+1}^{p-2} (\beta-\alpha_{i}) \right) (\beta +j \alpha_{1}+\alpha_{j+1}+\alpha_{j+2}+\cdot \cdot \cdot+\alpha_{p-2}).
\]
 From Corollary \ref{cor:h_k,1,p}, $$\frac{h^{j}_{k,1,p}}{(\beta - \alpha_{1})^{j-1}}$$ is a polynomial. From Theorem \ref{m}, 
$$\frac{h^{j}_{k,1,p}}{(\beta - \alpha_{1})^{j-1}}$$  
is a $p-$Eisenstein polynomial with respect to $(\beta - \alpha_{1})$. 

Replacing $h^{j}_{k,1,p}$ with $h^{j}_{k,1,p}/(\beta - \alpha_{1})^{j-1}$, we see that $h_{k,1,p}^{j}$ is irreducible over $\Q$. As $j<p-1$, Lemma \ref{Eisen2} shows that the lowest degree homogeneous part of $h_{k,1,p}^{j}$ has at least one simple linear factor defined over $\Q$, for example: 
$$l:= (\beta +j \alpha_{1}+\alpha_{j+1}+\alpha_{j+2}+\cdot \cdot \cdot+\alpha_{p-2}).$$ 
We argue similarly as in the proof of Theorem \ref{irrc}. By irreducibility of $h_{k,1,p}^{j}$ over $\Q$, the lowest degree homogeneous part of any factor of $h_{k,1,p}^{j}$ over $\C$ must have $l$ as one of its factor. As $l$ is a simple factor of the lowest degree homogeneous part of $h_{k,1,p}^{j}$, the polynomial $h_{k,1,p}^{j}$ is irreducible over $\C$. 
\end{proof}

Theorem \ref{inbetween} is equivalent to Theorem \ref{theorem3}, mentioned in the introduction. Fix $k\geq 0, p>3, p$ prime. Theorem \ref{inbetween} provides us with a chain of quasi-affine subvarities of $\mathcal{M}_{p}$ of strictly decreasing dimensions.
\begin{equation}\label{chain}
    \mathcal{M}_{p} \supset \Sigma_{k,1,p} \supset \Sigma_{2} \supset \Sigma_{3} \supset \cdot \cdot \cdot \supset \Sigma_{p-2} \supset \text{ \{unicritical points\} }.
\end{equation}
As $\mathcal{M}_{p}$ is a $(p-1)$-dimensional affine variety, dimension decreases by exactly $1$ in each stage of the above chain. Let $2 \leq j \leq p-2$. Similar to the separably critical case, as $\Sigma_{j}$ is irreducible, the set of all points of $\Sigma_{k,1,p}$ for which ramification index is exactly $j+1$ is open and dense in $\Sigma_{j}$. This shows that,
\begin{corollary}\label{mixedcoro}
    Let $k\geq 0, p$ be an odd prime and $2\leq j \leq p-1$. The set of all points of $\mathcal{M}_{p}$ for which the marked critical point is strictly $(k,1)$-preperiodic and has ramification index $j$, is an irreducible quasi-affine variety of dimension $p-j$.
\end{corollary}

Theorem \ref{inbetween} provides evidence for extending Milnor's conjecture to the mixed critical case. More precisely,
\begin{conjecture}\label{mixedconj1}
    Fix $k\geq 0,\;n>0,\; d>2,\; 2 < j < d$. The family of conjugacy classes of degree $d$ polynomials for which the marked critical point is strictly $(k,n)$-preperiodic, and has ramification index $\geq j$, is an irreducible quasi-affine subvariety of $\mathcal{M}_{d}$.
\end{conjecture}
\begin{remark}\label{ext}
    Observe that, the method above can be generalised for any $\mathcal{M}_{d}$, $d\in \N, d>2$. Let $k\in \N \cup \{0\}, n \in \N$. Let us assume that $h_{k,n,d}$ is irreducible over $\C$. Then, one can prove irreducibility in the mixed critical cases, by showing the irreducibility of $h_{k,n,d}(\alpha_{1},\cdot \cdot \cdot,\alpha_{1},\alpha_{j+1},\cdot \cdot \cdot,\alpha_{d-2},\beta)$ over $\C$, for $1<j<d-1$.  
\end{remark}

\section{The unicritical case}\label{unicritical} Putting $\alpha_{1}=\alpha_{2}=\cdot \cdot \cdot=\alpha_{d-1}=0$ in Equation \eqref{b}, we obtain a normal form for degree $d \geq 2$ unicritical polynomials,
\begin{equation}\label{unipoly}
    f(z)= z^{d} + \beta.
\end{equation}
Let $k, n\in \Z, k\geq 0, n>0$. Let $R_{k,n,d}$ be the polynomial in $\Q[\beta]$, whose roots are exactly the values of $\beta$ for which $0$ is strictly $(k,n)$-preperiodic under $f(z)= z^{d}+\beta$. \emph{Milnor's conjecture} in the unicritical case (Conjecture \ref{uniconj}), is equivalent to $R_{k,n,d}$ polynomials being either constant or irreducible over $\Q$, for every $k\geq 0, n>0.$  

Putting $\alpha_{1}=\alpha_{2}=\cdot \cdot \cdot=\alpha_{d-1}=0$ in Equation \eqref{g'}, we get that $R_{k,n,d}$ divides $h_{k,n,d}(0,0,\cdot \cdot \cdot,0,\beta)$ in $\Q[\beta]$. In this section, we will prove the irreducibility of $R_{k,1,p}$ over $\Q$, where $k\geq 0, p$ an odd prime.

\begin{theorem}\label{uni}
    For any $k,p\in \Z, k\geq 0, p$ odd prime, The polynomial $R_{k,1,p}$ is either constant or irreducible over $\Q$. 
\end{theorem}
\begin{proof}
The polynomial $R_{0,1,p}$ is $\beta$, hence irreducible. For any $n\in \N$, if $0$ is $(1,n)$-preperiodic under $f$ as in Equation \eqref{unipoly}, then it is also $n$-periodic. Hence, the polynomial $R_{1,1,p}$ is constant. 

Let $k\geq 2$. From section \ref{t'''}, the polynomial $h_{k,1,p}$ is $p$-Eisenstein with respect to the polynomial $\beta-\alpha_{1}$, and its lowest degree homogeneous part is 
$$p(\beta + \sum_{i=1}^{p-2} \alpha_{i})\prod_{i=2}^{p-2} (\beta - \alpha_{i}).$$ 
Hence, the lowest degree term of $h_{k,1,p}(0,0,\cdot \cdot \cdot,0,\beta)$  is $p\beta^{p-2}$, and the polynomial 
$$\dfrac{h_{k,1,p}(0,0,\cdot \cdot \cdot,0,\beta)}{\beta^{p-2}},$$
is $p$-Eisenstein with respect to $\beta$. By definition of $R_{k,1,p}$, $\;\beta$ divides $R_{k,1,p}$ iff $k=0$. Hence, for any $k\geq 2$, if $R_{k,1,p}$ is non-constant polynomial, then it is $p$-Eisenstein with respect to $\beta$, for every $k\in \N$. Therefore $R_{k,1,p}$ is either constant or irreducible over $\Q$, for any $k\in \Z_{\geq 0}$. 
\end{proof}
\begin{remark}
    Theorem \ref{uni} is equivalent to Theorem \ref{theorem2} stated in the introduction. A different proof of this theorem can be found in  \cite[Corollary~1.1.i]{10.7169/facm/1799}. A parallel but weaker version of Theorem \ref{uni}, obtained by choosing the normal form of degree $d$ unicritical polynomial to be $g(z)=az^{d} +1$, is proved by Buff, Epstein and Koch in \cite[Theorem~22]{article1}.
\end{remark}

\section{A study of $\Sigma_{k,2,p}$}\label{obs1}
For this section, $f$ is as defined in Equation \eqref{fdegp}.
$$f(z)=z^{p} + \sum_{i=2}^{p-1}(-1)^{i}\frac{p\cdot s_{i}(\bar{\alpha})}{p-i} z^{p-i} - \alpha_{1}^{p} - \sum_{i=2}^{p-1}(-1)^{i}\frac{p\cdot s_{i}(\bar{\alpha})}{p-i} \alpha_{1}^{p-i} + \beta.$$

\begin{lemma}\label{h02p}
The polynomial $h_{0,2,p}(\text{mod }p)$ is reducible in $\F_{p}[\alpha_{1},\alpha_{2},\cdot \cdot \cdot,\alpha_{p-2},\beta]$, for all prime $p>3$.
\end{lemma}
\begin{proof}
We have, 
\begin{align*}
    f_{0,2,p}& =f^{2}(\alpha_{1})-\alpha_{1}\\
    &= (\beta^{p}-\alpha_{1}^{p})+ \sum_{i=2}^{p-1}(-1)^{i}\frac{p\cdot s_{i}(\bar{\alpha})}{p-i}(\beta^{p-i}-\alpha_{1}^{p-i})+(\beta-\alpha_{1})\\
&\equiv (\beta-\alpha_{1})^{p}+(\beta-\alpha_{1})\;\modp.
\end{align*}
As $f_{0,1,p}= \beta - \alpha_{1}$, we have
\begin{equation*}
    h_{0,2,p}\equiv (\beta-\alpha_{1})^{p-1}+1\;\modp.
\end{equation*}
For $p\equiv 1(\text{mod } 4)$, $-1$ being a square modulo $p$, $h_{0,2,p}$ is reducible over $\F_{p}$. For $p\equiv 3(\text{mod } 4), p\neq 3$, the integer $p-1$ has at least one odd prime factor, say $q$. Then $(\beta-\alpha_{1})^{q}+1$ divides $h_{0,2,p}$ modulo $p$.
So, $h_{0,2,p}\;(\text{mod }p)$ is reducible for any prime $p>3$.
\end{proof}
\begin{remark}
 Lemma \ref{h02p} shows that modulo $p$, the polynomial $h_{0,2,p}$ is reducible, for any prime $p>3$. So, one can not extend the proof for $(k,1,p)$ case directly to this case.
\end{remark}
\section{A study of $\Sigma_{k,1,p^{e}}$}\label{obs2}
Let $p,e\in \N, p$ odd prime. Putting $d=p^{e}$ in Equation \eqref{b}, the expression of $f$ becomes,
\begin{equation*}
    f(z)= z^{p^{e}} + \sum_{i=2}^{p^{e}-1}(-1)^{i}\frac{p^{e}}{p^{e}-i}\cdot s_{i}(\bar{\alpha})\cdot z^{p^{e}-i} - \alpha_{1}^{p^{e}} - \sum_{i=2}^{p^{e}-1}(-1)^{i}\frac{p^{e}}{p^{e}-i}\cdot s_{i}(\bar{\alpha})\cdot \alpha_{1}^{p^{e}-i} + \beta.
\end{equation*}
Here, $f\in \Z_{(p)}[z,\alpha_{1},\alpha_{2},\cdot \cdot \cdot,\alpha_{p^{e}-1},\beta], \hat{f}_{k,1,p^{e}},h_{k,1,p^{e}}\in \Z_{(p)}[\alpha_{1},\alpha_{2},\cdot \cdot \cdot,\alpha_{p^{e}-2},\beta]$. One can prove the following lemmas similarly as they were proved in the $(k,1,p)$ case.
\begin{lemma}
The polynomial $h_{0,1,p^{e}}$ is $\beta-\alpha_{1}$.
\end{lemma}
\begin{lemma}\label{powermod}
For any $k\in \N$, we have $h_{k,1,p^{e}}\equiv h_{0,1,p^{e}}^{N_{k,p^{e}}}(\text{mod }p)$, for some $N_{k,p^{e}}\in \N$.
\end{lemma}
\begin{lemma}\label{respe}
For any $k\in \N, k\geq 2,$ the following equality holds, 
$$h_{k,1,p^{e}}=\frac{\hat{f}_{k,1,p^{e}}}{\hat{f}_{k-1,1,p^{e}}(\beta-\alpha_{1})}$$ 
and the lowest degree homogeneous part of $h_{k,1,p^{e}}$ is 
$$p^{e}(\beta + \sum_{i=1}^{p^{e}-2} \alpha_{i})\prod_{i=2}^{p^{e}-2}(\beta-\alpha_{i})$$.
\end{lemma}
\begin{remark}
Due to Lemma \ref{respe}, we get that the resultant 
\begin{align*}
\mbox{Res}(h_{k,1,p^{e}},h_{0,1,p^{e}}) &=p^{e}(2\alpha_{1}+\alpha_{2}+\cdot \cdot \cdot+\alpha_{p^{e}-2})\prod_{i=2}^{p^{e}-2}(\alpha_{1}-\alpha_{i})\\
&\equiv 0\;(\text{mod }p^{2}),
\end{align*} 
for any $e\in \N, e\geq 2$. This shows the obstruction to directly generalise the proof of irreducibility of $\Sigma_{k,1,p}$ directly to the case of $\Sigma_{k,1,p^{e}}, e\geq 2$. 
\end{remark}

\end{document}